\documentclass{amsart}
\usepackage[utf8]{inputenc}

\usepackage{csquotes}
\usepackage{float}
\usepackage{tikz-cd}
\usepackage{amsmath}%
\usepackage{amsfonts}%
\usepackage{amsthm}
\usepackage{amssymb}%
\usepackage{graphicx}
\usepackage{hyperref}
\usepackage[all,cmtip]{xy}
\usepackage{tabularx}
\usepackage{colonequals}
\usepackage{geometry}
\theoremstyle{plain}
\newtheorem{theorem}{Theorem}[section]

\newtheorem{lemma}[theorem]{Lemma}
\newtheorem{corollary}[theorem]{Corollary}

\newtheorem{proposition}[theorem]{Proposition}


\theoremstyle{definition}

\newtheorem{definition}[theorem]{Definition}

\theoremstyle{remark}
\newtheorem{remark}[theorem]{Remark}
\makeatletter
\newcommand{\bigplus}{%
  \DOTSB\mathop{\mathpalette\mattos@bigplus\relax}\slimits@
}
\newcommand\mattos@bigplus[2]{%
  \vcenter{\hbox{%
    \sbox\z@{$#1\sum$}%
    \resizebox{!}{0.9\dimexpr\ht\z@+\dp\z@}{\raisebox{\depth}{$\m@th#1+$}}%
  }}%
  \vphantom{\sum}%
}
\makeatother

\newcommand{\Z}{\mathbb{Z}}
\newcommand{\Q}{\mathbb{Q}}

\newcommand{\CC}{\mathbb{C}}

\newcommand{\om}{\omega}

\renewcommand{\P}{\mathbb{P}}
\newcommand{\F}{\mathbb{F}}

\newcommand{\al}{\alpha}

\renewcommand{\O}{\mathcal{O}}

\newcommand{\dd}{\mathrm{d}}
\newcommand{\ep}{\epsilon}

\newcommand{\pp}{\mathfrak{p}}

\newcommand{\qq}{\mathfrak{q}}

\renewcommand{\div}{\mathrm{div}}

\begin{document}
\title{Elliptic curves over totally real quartic fields not containing $\sqrt{5}$ are modular}
\author{Josha Box}
\maketitle
\begin{abstract}
    We prove that every elliptic curve defined over a totally real number field of degree 4 not containing $\sqrt{5}$ is modular. To this end, we study the quartic points on four modular curves. 
\end{abstract}
\section{Introduction}
\label{quartic points}

Since the major breakthrough of Wiles \cite{wiles} and Taylor--Wiles \cite{taylorwiles} on semi-stable elliptic curves and the subsequent proof by Breuil, Conrad, Diamond and Taylor \cite{bcdt} in 2000 that all elliptic curves over $\Q$ are modular, the natural follow-up question was: can we extend these results to elliptic curves over other number fields? In this article we do so for all totally real quartic fields not containing a root of 5, by combining a variety of computational methods to study the quartic points on multiple modular curves. Crucial is the use of the ``partially relative'' symmetric Chabauty method developed by the author, Gajovi\'c and Goodman in \cite{bgg}. 

Let $K$ be a totally real number field. We say that an elliptic curve $E$ over $K$ of conductor $\mathcal{N}$ is \emph{modular} when there is a Hilbert newform $\mathfrak{f}$ of level $\mathcal{N}$, parallel weight 2 and with rational Hecke eigenvalues, such that their corresponding systems of compatible  $\ell$-adic Galois representations are isomorphic. Here the $\ell$-adic Galois representation attached to $E$ arises from the action of $\mathrm{Gal}(\overline{K}/K)$ on the Tate module $T_{\ell}(E)$, and the Galois representation associated to $\mathfrak{f}$ was defined by Taylor \cite{taylor}, extending the construction of Eichler and Shimura for $K=\Q$ (see e.g.\ \cite{diamondshurman}). 

The first modularity results beyond $\Q$ were proved by Jarvis and Manoharmayum \cite{jarvis}, who showed that all semi-stable elliptic curves over $\Q(\sqrt{2})$ and $\Q(\sqrt{17})$ are modular. A real leap forward was then made in 2014 by Freitas, Le Hung and Siksek \cite{freitas} when they proved modularity of elliptic curves over all real quadratic fields simultaneously. More recently, Derickx, Najman and Siksek \cite{derickx} extended this to elliptic curves over all totally real cubic fields. We go one step further to totally real quartic fields, albeit only those not containing a root of 5. In Section \ref{furtherstudy}, we discuss some of the difficulties one encounters for quartic fields containing $\sqrt{5}$ and for quintic fields.
\begin{theorem}\label{mainthm}Let $E$ be an elliptic curve over a totally real quartic number field not containing a square root of 5. Then $E$ is modular.
\end{theorem}
Our strategy for proving modularity is similar to that of Wiles \cite{wiles}, Freitas--Le Hung--Siksek \cite{freitas} and Derickx--Najman--Siksek \cite{derickx}. We make use of strong modularity lifting theorems due to Breuil--Diamond \cite{bd}, Thorne \cite{thorne} and Kalyanswami \cite{kalyanswamy} (building on the work of many others), which \emph{almost} prove modularity over all totally real number fields, in the following sense:  the mod $p$ Galois representation of a non-modular elliptic curve over a totally real number field must have ``small'' image for $p\in \{3,5,7\}$. This particular meaning of ``small'' can be parametrised precisely by modular curves, leading to the following corollary of the aforementioned modularity lifting results. 

\begin{theorem}\label{thm1.1}
Suppose that $E$ is an elliptic curve over a totally real number field $K$ satisfying $K\cap \Q(\sqrt{5})=\Q$ and $K\cap \Q(\zeta_7)=\Q$, where $\zeta_7$ is a primitive 7th root of unity. If $E$ is not modular, then $E$ gives rise to a $K$-point on one of the following modular curves:
\[
X(\mathrm{b}3,\mathrm{b}5,\mathrm{b}7),\quad X(\mathrm{s}3,\mathrm{b}5,\mathrm{b}7),\quad X(\mathrm{b}3,\mathrm{b}5,\mathrm{e}7),\quad X(\mathrm{s}3,\mathrm{b}5,\mathrm{e}7).
\]
\end{theorem}
In Section \ref{modliftsec} we define these curves and show how this theorem follows from modularity lifting results. Theorem \ref{thm1.1} reduces the modularity question to studying $K$-rational points on these four curves. In the classical case $K=\Q$, it sufficed to forget about the prime 7 and consider only the  curves $X(\mathrm{b}3,\mathrm{b}5)$ and $X(\mathrm{s}3,\mathrm{b}5)$. Both of these are ellipic curves with finite Mordell--Weil group, and the rational points were swiftly found to correspond to cusps or modular elliptic curves. When considering elliptic curves over totally real fields of degree $d>1$, this final piece of the puzzle, i.e.\ finding all degree $d$ points on these four curves (or their quotients), ceases to be a trivial exercise.  In the quadratic case, Freitas, Le Hung and Siksek in fact had to study seven modular curves, as they did not yet have the results of Thorne \cite{thorne} and Kalyanswamy \cite{kalyanswamy} at their disposal. In Remark \ref{quadrmk}, we describe a shorter proof of modularity of elliptic curves over real quadratic fields now made possible by these stronger modularity lifting theorems.

We prove Theorem \ref{mainthm} by studying the quartic points on these four curves. Determining all quartic points on either of those curves directly is not computationally feasible: the curves have genera 13, 21, 73 and 153 respectively. The saving grace here is the fact that all four curves give rise to a rich tree of quotient curves of smaller genus, on which explicit computations can be performed. For none of the four curves, however, did it suffice to study the quartic points on a single such quotient curve; instead, we combine information from multiple quotient curves using their maps to common quotients further down the tree. For $X(\mathrm{b}3,\mathrm{b}5,\mathrm{b}7)$, we work on 10 different quotients by Atkin--Lehner involutions of genus up to 5, whereas for $X(\mathrm{s}3,\mathrm{b}5,\mathrm{b}7)$ we consider two genus 3 quotients.

To study the final two curves, we need two important new  ingredients:  the relative symmetric Chabauty method developed by the author, Gajovi\'c and Goodman \cite{bgg}, and the algorithm for determining models for quotients of modular curves found by the author in \cite{boxmodels}. These are applied to quotient curves of genus 5, 6 and 8.

\subsection{Modular curves} In this section, we briefly mention some important basic facts about modular curves. 

Let $N$ be a positive integer and consider $G\subset \mathrm{GL}_2(\Z/N\Z)$. Let $\zeta_N$ be a primitive $N$-th root of unity, and let $R_G$ be the subring of $\Z[1/N,\zeta_N]$ fixed by the action of $\mathrm{det}(G)\subset (\Z/N\Z)^{\times}$ on $\zeta_N$. To $G$ we associate an $R_G$-scheme $X_G$ called a \emph{modular curve}, defined in \cite{katzmazur} and \cite{deligne} as the compactification of a coarse moduli scheme $Y_G$ parametrising elliptic curves with ``$G$-level structure''. See also \cite{boxmodels} for more details on this construction. 

Let $\mathcal{H}$ be the complex upper half plane, and define $\Gamma_G$ to be the inverse image of $G\cap \mathrm{SL}_2(\Z/N\Z)$ under $\mathrm{SL}_2(\Z)\to \mathrm{SL}_2(\Z/N\Z)$. This \emph{congruence subgroup} $\Gamma_G$ acts on $\mathcal{H}$ by fractional linear transformations, and $(Y_G)_{\CC}\simeq \Gamma_G\backslash \mathcal{H}$. When $G\subset H$, we obtain a morphism $X_G\to X_H\times_{R_H}R_G$ which is a map of curves of degree $[\pm \Gamma_G:\pm\Gamma_H]$. The curve $X(1)\colonequals X_{\mathrm{GL}_2(\Z/N\Z)}$ is independent of the choice of $N$, and the inclusion $G\subset \mathrm{GL}_2(\Z/N\Z)$ gives rise to a map $j:\;X_G\to X(1)=\P^1$ called the $j$-map. The \emph{cusps} are those $Q\in X_G$ such that $j(Q)=\infty$, and $X_G(\overline{\Q})\setminus Y_G(\overline{\Q})$ is exactly the set of cusps. 

Suppose now that $-I\in G$ and $\mathrm{det}(G)=(\Z/N\Z)^{\times}$. Then $X_G$ is a $\Z[1/N]$-scheme, and we consider it as a curve over $\Q$.  For an elliptic curve $E$ over a number field $K$, we denote by $\overline{\rho}_{E,N}$ its mod $N$ representation $\mathrm{Gal}(\overline{K}/K)\to \mathrm{GL}_2(\Z/N\Z)$, defined up to conjugation in $\mathrm{GL}_2(\Z/N\Z)$. Being a \emph{coarse} moduli scheme, $Y_G(L)$ parametrises pairs $(E,[\phi]_G)$ defined over $L$ a priori only for algebraically closed fields $L$. In this special case, however, we can say more about the moduli interpretation over $K$. 
\begin{itemize}
    \item If $E$ is an elliptic curve over $K$ such that $\mathrm{Im}(\overline{\rho}_{E,N})\subset G$ up to conjugation, then there exists $Q\in Y_G(K)$ with $j(Q)=j_E$.
    \item Conversely, if $Q\in Y_G(K)$ satisfies $j(Q)\notin \{0,1728\}$, there exists an elliptic curve $E/K$ with $j$-invariant $j(Q)$ such that $\mathrm{Im}(\overline{\rho}_{E,N})\subset G$ up to conjugation.
\end{itemize}
This is well-known and dates back to Deligne and Rapoport \cite{deligne}; see e.g.\ \cite[Proposition 3.2]{zywina2} for a proof.

When $N,M$ are coprime, $G\subset \mathrm{GL}_2(\Z/N\Z)$ and $H\subset \mathrm{GL}_2(\Z/M\Z)$, we obtain a new group $F\subset \mathrm{GL}_2(\Z/NM\Z)$ as the intersection of the inverse images of $G$ and $H$ under the reduction maps. As curves over $\Q$, $X_F$ is then the normalisation of $X_H\times_{X(1)}X_G$. Since the $j$-map on modular curves is ramified only at points $Q$ with $j(Q)\in \{\infty,0,1728\}$, we find away from those $j$-invariants that $Q\in (X_H\times_{X(1)}X_G)(K)$ if and only if $Q\in X_F(K)$.

When $q$ is the highest power of a prime $p$ that divides $N$, and the image of $G$ in $\mathrm{GL}_2(\Z/q\Z)$ is the Borel subgroup $B(q)$, the curve $X_G/\Q$ admits an \emph{Atkin--Lehner involution} $w_q: \; X_G\to X_G$. An important fact is that for $Q\in Y_G(K)$, the two $j$-invariants $j(Q)$ and $j(w_q(Q))$ correspond to $K$-isogenous elliptic curves. For powers $q,q'$ of distinct primes, the Atkin--Lehner involutions $w_q$ and $w_{q'}$ commute, and we denote their product by $w_{qq'}=w_qw_{q'}$.

Finally, we denote by $S_k(\Gamma,K)$ the space of weight $k$ cusp forms with respect to the congruence subgroup $\Gamma$ and with Fourier coefficients in $K$. For brevity, we write $S_k(\Gamma)\colonequals S_k(\Gamma,\CC)$.  
\subsection{Consequences of modularity lifting theorems}\label{modliftsec}
We mention three consequences of modularity lifting theorems, and how they lead to Theorem \ref{thm1.1}. Let $p$ be a prime number. We denote by $B(p)\subset \mathrm{GL}_2(\F_p)$ the Borel subgroup, by $C_{\mathrm{s}}^+(p)\subset \mathrm{GL}_2(\F_p)$ the normaliser of a split Cartan subgroup, and by $C_{\mathrm{ns}}^+(p)\subset \mathrm{GL}_2(\F_p)$ the normaliser of a non-split Cartan subgroup. Finally, we also consider
\[
G(\mathrm{e}7)\colonequals \left\langle \begin{pmatrix} 0 & 5 \\ 3 & 0 \end{pmatrix}, \begin{pmatrix} 5 & 0 \\ 3 & 2\end{pmatrix}\right\rangle\subset \mathrm{GL}_2(\F_7),
\]
which is an index 2 subgroup of a  $C_{\mathrm{ns}}^+(7)$. 
\begin{theorem}\label{modliftthm}
Suppose that $E$ is a non-modular elliptic curve over a totally real field $K$. Then
\begin{itemize}
\item[(i)]  $\mathrm{Im}(\overline{\rho}_{E,3})$ is conjugate to a subgroup of $C_s^+(3)$ or $B(3)$,
\item[(ii)] if $\sqrt{5}\notin K$, then $\mathrm{Im}(\overline{\rho}_{E,5})$ is conjugate to a subgroup of $B(5)$, and
\item[(iii)] if $K\cap \Q(\zeta_7)=\Q$, then $\mathrm{Im}(\overline{\rho}_{E,7})$ is conjugate to a subgroup of $B(7)$ or $G(\mathrm{e}7)$. 
\end{itemize}
\end{theorem}
\begin{proof}
Part (i) is a consequence of modularity lifting theorems due amongst others to Breuil and Diamond \cite{bd}, which show that the restriction of $\overline{\rho}_{E,3}$ to $\mathrm{Gal}(\overline{K}/K(\zeta_3))$ is absolutely reducible. See \cite[Theorem 3]{freitas} for more details. It then follows from \cite[Proposition 6]{rubin} that $\overline{\rho}_{E,3}$ is conjugate to a subgroup of $C_s^+(3)$ or $B(3)$. 

Part (ii) was shown by Thorne \cite{thorne}. In \cite[Proposition 4.3 and Theorem 4.4]{kalyanswamy}, Kalyanswami shows that if $\zeta_7+\zeta_7^{-1}\notin K$, then $\mathrm{Im}(\overline{\rho}_{E,7})$ is conjugate to a subgroup of either $B(7)$ or $C_{\mathrm{ns}}^+(7)$. In the latter case, Freitas, Le Hung and Siksek study the subgroups of $C_{\mathrm{ns}}^+(7)$ and show in \cite[Proposition 4.1 (c)]{freitas} that in fact the image must be contained in the index 2 subgroup $G(\mathrm{e}7)$.
\end{proof}
We say that an elliptic curve $E$ over a number field $K$ is a $\Q$-curve, when $E$ is $\overline{K}$-isogenous to each of its Galois conjugates. Another important modularity result we shall need is the following.
\begin{theorem}[Ribet \cite{ribet}]
Every $\Q$-curve is modular. 
\end{theorem}
For a prime $p$, we define $X(\mathrm{b}p)$, $X(\mathrm{s}p)$, $X(\mathrm{ns}p)$ and $X(\mathrm{e}p)$ to be the curves $X_G$, where $G$ is $B(p)$, $C_{\mathrm{s}}^+(p)$, $C_{\mathrm{ns}}^+(p)$ and $G(\mathrm{e}7)$ respectively. We note that $X(\mathrm{b}p)=X_0(p)$. For distinct primes $p_1,\ldots,p_n$ and $\mathrm{u}_1,\ldots,\mathrm{u}_n\in \{\mathrm{b},\mathrm{s},\mathrm{ns},\mathrm{e}\}$, we define 
\[
X(\mathrm{u}_1p_1,\ldots,\mathrm{u}_np_n)
\]
to be the normalisation of $X(\mathrm{u}_1p_1)\times_{X(1)}\cdots \times_{X(1)}X(\mathrm{u}_np_n)$. 

Finally, let $E/K$ be a non-modular elliptic curve over a totally real field. Theorem \ref{modliftthm} and the discussion in the previous section imply that $E$ gives rise to a $K$-point on  $X(\mathrm{b}3,\mathrm{b}5,\mathrm{b}7)$, $X(\mathrm{s}3,\mathrm{b}5,\mathrm{b}7)$, $X(\mathrm{b}3,\mathrm{b}5,\mathrm{e}7)$ or $X(\mathrm{s}3,\mathrm{b}5,\mathrm{e}7)$, which proves Theorem \ref{thm1.1}. 

It thus suffices to find the quartic points on those curves. In fact, it suffices to find the quartic points $P$ whose $j$-invariant $j(P)$ also has degree 4: if smaller, then $P$ is supported on an elliptic curve $E$ that either has $j$-invariant 0 or 1728 and hence is a $\Q$-curve, or is a quadratic twist $E=E'\otimes \chi$, where $\chi$ is a quadratic character and $E$' is an elliptic curve defined over a real number field of degree 1 or 2. In the latter case, $E'$ is known to be modular. If $E'$ corresponds to the Hilbert modular form $\mathfrak{f}$, then $E$ is also modular and corresponds to $\mathfrak{f}\otimes \chi$. 
\begin{theorem}\label{thm1.5}
On $X(\mathrm{b}3,\mathrm{b}5,\mathrm{b}7)=X_0(105)$ and $X(\mathrm{s}3,\mathrm{b}5,\mathrm{b}7)$, all quartic points with quartic $j$-invariant are $\Q$-curves. On $X(\mathrm{b}3,\mathrm{b}5,\mathrm{e}7)$ and $X(\mathrm{s}3,\mathrm{b}5,\mathrm{e}7)$, all quartic points with quartic $j$-invariant are either $\Q$-curves, or have a non-totally real $j$-invariant displayed in Table \ref{tableb5ns7}.
\end{theorem}
The remainder of this article is devoted to the proof of this theorem.

\subsection{Overview}
We first give, in Section \ref{methodssec}, a summary of the various methods used to find degree $4$ points on modular curves. We then study the quartic points on $X(\mathrm{b}3,\mathrm{b}5,\mathrm{b}7)$ in Section \ref{X0105sec} by making use of ten of its quotients, after which we study the quartic points on $X(\mathrm{s}3,\mathrm{b}5,\mathrm{b}7)$ in Section \ref{Xs3b5b7sec} using the two genus 3 hyperelliptic quotients $X(\mathrm{b}5,\mathrm{b}7)$ and $X(\mathrm{s}3,\mathrm{b}7)$. We then use the Chabauty--Coleman method developed in \cite{bgg} to study the quartic points on $X(\mathrm{b}5,\mathrm{ns}7)$ in Section \ref{Xb5ns7sec}. Finally, we use the knowledge of the (infinitely many) quartic points on $X(\mathrm{b}5,\mathrm{ns}7)$ to study the quartic points on $X(\mathrm{b}3,\mathrm{b}5,\mathrm{e}7)$ and $X(\mathrm{s}3,\mathrm{b}5,\mathrm{e}7)$ in Section \ref{X1andX2sec}, using a combination of Chabauty's method on two quotient curves with a Mordell--Weil sieve. 
%

The explicit computations on these curves were done in \texttt{Magma}, and are publicly available at 
\[
\texttt{\href{https://github.com/joshabox/quarticpoints/}{https://github.com/joshabox/quarticpoints/}} \;.
\]
\subsection{Acknowledgements}
The author would like to heartily thank Samir Siksek for many inspiring conversations, countless  valuable comments and suggestions, and for his support and optimism throughout this project.

The author is also grateful to Andrew Sutherland for a helpful correspondence.
\section{Methods for determining the degree $d$ points on a (modular) curve}\label{methodssec} To determine their quartic points, we use a myriad of methods on a sizeable number of quotients of the four curves of Theorem \ref{thm1.1}. In this section, we provide an overview.
\subsection{The symmetric power}\label{sympowsec}
Computing the $K$-rational points on a curve $X/\Q$ for all number fields $K$ of degree $d>1$ simultaneously may seem like a daunting exercise. The trick is to study instead the rational points on the symmetric power
\[
X^{(d)}\colonequals \mathrm{Sym}^d(X).
\]
Those rational points $X^{(d)}(\Q)$ correspond 1-1 with the effective $\Q$-rational divisors of degree $d$ on $X_{\overline{\Q}}$, and we shall denote them as such. For each $Q\in X(\overline{\Q})$ defined over a number field of degree $d$, the sum of the Galois conjugates of $Q$ is such a divisor, so it suffices to study $X^{(d)}(\Q)$ instead. While $X^{(d)}$ is a $d$-dimensional variety, its Albanese variety is the Jacobian $J(X)$ of $X$. For each $\Q$-rational degree $d$ divisor $D_0$ on $X$, we thus obtain an Abel--Jacobi map
\[
\iota_{D_0}:\;X^{(d)}(\Q)\longrightarrow J(X)(\Q),\quad D\mapsto [D-D_0].
\]
Many classical tools for studying the rational points on a curve  $X$ via its Mordell--Weil group can now also be used to study $X^{(d)}(\Q)$, c.f.\ Sections \ref{finitemwgpsec} and \ref{infmwgpsec}.

\subsection{Computing models for modular curves}\label{modelssec}
The first important ingredient for carrying out computations on a curve is a set of defining equations. For modular curves, these can be obtained from the $q$-expansions of the corresponding spaces of cusp forms, following Galbraith \cite{galbraith}.

For curves of the form $X_0(N)$, this has been implemented in the  Modular Curves and Small Modular Curves packages in \texttt{Magma}, which we make use of. For other modular curves, such as $X(\mathrm{b}5,\mathrm{ns}7)$ and $X(\mathrm{b}3,\mathrm{ns}7)$, we apply the algorithm developed by the author in \cite{boxmodels}. 

The Small Modular Curves package also contains an implementation for computing the maps $X_0(N)\to X_0(M)$ when $M\mid N$, for small values of $N$ only. For larger values of $N$, we determine such maps using code written by \"Ozman and Siksek in \cite{ozman}. For other modular curves, we use \cite{boxmodels} to find $q$-expansions for the corresponding space of cusp forms and use these to determine the morphisms, such as $X(\mathrm{b}3,\mathrm{ns}7)\to X(\mathrm{ns}7)$, following the method outlined in \cite[Section 3]{ozman}. Let us give some more details. 

Suppose that we would like to compute explicitly a morphism $\pi:\;X\to Y$ of modular curves, and we have models for $X\subset \P^n$ and $Y\subset \P^m$. Denote by $x_0,\ldots,x_n$ and $y_0,\ldots,y_m$ the coordinates in $\P^n$ and $\P^m$ respectively. Suppose that we know $q$-expansions for $y_0/y_m,\ldots,y_{m-1}/y_m\in \Q(Y)$ and $x_0/x_n,\ldots,x_{n-1}/x_n\in \Q(X)$. Given a degree $d\geq 1$, we can compute the $q$-expansions of all monomials of degree $d$ in $x_0/x_n,\ldots,x_{n-1}/x_n$. Using linear algebra, we can thus attempt to find for each $i\in \{0,\ldots,m-1\}$ two polynomials $p_i,r_i$ such that
\[
p_i\left(\frac{x_0}{x_n}(q),\ldots,\frac{x_{n-1}}{x_n}(q)\right)=\frac{y_i}{y_m}(q)\cdot r_i\left(\frac{x_0}{x_n}(q),\ldots,\frac{x_{n-1}}{x_n}(q)\right),
\]
at least up to some high order $q^B$. After multiplying $p_i$ and $r_i$ by the same polynomial, we may assume that $r_1=\ldots=r_{m-1}$; we call this common value $r$. Having found such polynomials, we can check whether
\[
\pi':\;(x_0:\ldots:x_n)\mapsto (P_1(x_0,\ldots,x_n):\ldots: P_{m-1}(x_0,\ldots,x_n):R(x_0,\ldots,x_n)),
\]
defines a morphism $\pi':X\to Y$, where $P_1,\ldots,P_{m-1},R$ are the homogenisations of $p_1,\ldots,p_{m-1},r$ respectively. We can check that $\pi=\pi'$ when we know $\mathrm{deg}(\pi)$. We first verify that $\mathrm{deg}(\pi)=\mathrm{deg}(\pi')$. Then, we use the coordinate maps $\pi_i:\; Y\to \P^1,\; (y_0:\ldots:y_m)\mapsto (y_i:y_m)$ for $i\in \{0,\ldots,m-1\}$. Now $\pi_i\circ \pi$ and $\pi_i\circ \pi'$ correspond to functions in $\Q(X)$ whose difference $f_i$ has divisor with valuation at least $B$ at the infinity cusp. Looking at the poles, we see that $\mathrm{deg}(\mathrm{div}(f_i))\leq 2\mathrm{deg}(\pi)\mathrm{deg}(\pi_i)$. If $B> 2 \mathrm{deg}(\pi)\cdot \mathrm{deg}(\pi_i)$ for each $i$, then we must have $f_i=0$ for each $i$ and thus $\pi=\pi'$. 
\subsection{Computing generators of Mordell--Weil groups}
In general, computing Mordell--Weil groups of Jacobians of curves is an unsolved problem. For (modular) curves of small genus, however, a solution can often be found. 

\subsubsection{The rank}\label{ranksec}
We first discuss the rank. For modular curves of the form $X(\mathrm{u}_1p_1,\ldots,\mathrm{u}_np_n)$ with $p_1,\ldots,p_n$ distinct primes and $\mathrm{u}_1,\ldots,\mathrm{u}_n\in \{\mathrm{b},\mathrm{s},\mathrm{ns}\}$, it can often be determined whether the rank is positive or not. We denote the Jacobian of such a curve by $J(\mathrm{u}_1p_1,\ldots,\mathrm{u}_np_n)$. The curve $X(\mathrm{s}p)$ is isomorphic to $X_0(p^2)/w_{p^2}$ (see e.g.\ \cite[p.\ 555]{conrad} or \cite[Example 2.10]{boxmodels}), while $J(\mathrm{ns}p)$ is isogenous to the new part of $J(X_0(p^2)/w_{p^2})$ (see \cite[Theorem 1]{chen}). We can therefore use an algorithm of Stein \cite{stein} implemented in the Modular Curves package in \texttt{Magma} to identify Hecke eigenforms $f_1,\ldots,f_n \in S_2(\Gamma_1(N))$ for some $N$, such that 
\[
J(\mathrm{u}_1p_1,\ldots,\mathrm{u}_np_n)\sim A_{f_1}\times\cdots \times A_{f_n},
\]
where $A_f$ is the Abelian variety associated to $f$ by Eichler and Shimura. By Kolyvagin and Logach\"ev \cite{kolyvagin}, we then have $\mathrm{rk}(A_f(\Q))=0$ if and only if $L(f,1)\neq 0$, where $L(f,1)$ is the value of the $L$-function of $f$ at 1.  We apply this argument to all modular curves we consider: the curve $X_0(105)$ has Mordell--Weil group of rank 0, whereas the other three curves defined in Theorem \ref{thm1.1} have a positive rank Mordell--Weil group. If $L(f,1)$ is zero up to \texttt{Magma}'s precision, this strongly suggests, but does not prove, that the rank is positive. Even when positive, the information which factors $A_f$ do have rank zero can be invaluable. For example, $X(\mathrm{b}5,\mathrm{ns}7)$ and its ``sisters'' $X_1$ and $X_2$ defined in Section \ref{X1andX2sec} have rank at least 2, but we successfully applied Chabauty's method using knowledge of such rank zero quotients. 

When the modular curve $X$ is not only built up of $\mathrm{b},\mathrm{s}$ and $\mathrm{ns}$, one can still use the algorithm in \cite{boxmodels} to find a map $X_1(N)\to X$, defined in general over an abelian number field $K$. This yields a morphism with finite kernel $J(X)\to J_1(N)$ defined over $K$. Then the information on the ranks $\mathrm{rk}(A_f(K))$ for eigenforms $f\in S_2(\Gamma_1(N))$ can be used to find rank zero quotients of $J(X)$, as we do in Section \ref{X1andX2sec}.\\ 

 When $g$ is a Hecke eigenform, we saw that $\mathrm{rk}(A_g(\Q))=0$ when  $L(g,1)\neq 0$. How does one check whether $\mathrm{rk}(A_g(K))=0$ when $K\neq \Q$? We solve this for abelian number fields $K$, based on the work of Gonz\'alez-Giménez and Guitart \cite{guitart} and Guitart and Quer \cite{quer}. Suppose that $g=\sum a_n q^n\in S_2(\Gamma_1(N))$ is a newform with Nebentypus character $\chi$ and without complex multiplication. We introduce three fields associated to $g$. Define $E_g\colonequals \Q(\{a_n\})$ to be the Hecke eigenvalue field of $g$, and $F_g\colonequals \Q(\{a_p^2\chi(p)^{-1} \mid p\nmid N\})$. Then $F_g$ is totally real and $E_g/F_g$ is an abelian extension. When $\psi$ is a Dirichlet character, we denote by $g\otimes \psi$ the unique newform with Fourier coefficients $a_n(g\otimes \psi)=a_n(g)\psi(n)$ for all $n$ coprime to the level of $g$ and the conductor of $\psi$. For each $s\in \mathrm{Gal}(E_g/F_g)$, there is a unique Dirichlet character $\chi_s\colon \mathrm{Gal}(\overline{\Q}/\Q)\to \CC^{\times}$ such that $g^s=g\otimes \chi_s$ (see \cite[Section 3]{ribet2}). Here each $\chi_s$ factors via $\mathrm{Gal}(L_g/\Q)$, where $L_g\colonequals \cap_{s\in \mathrm{Gal}(E_g/F_g)} \overline{\Q}^{\mathrm{Ker}\chi_s}$. Associated to $g$ by Eichler and Shimura is the $\Q$-simple abelian variety $A_g/\Q$. Then, as noticed in \cite[below Proposition 1]{guitart}, $A_g$ is $L_g$-isogenous to a power $B^n$ of a $\overline{\Q}$-simple abelian variety $B/L_g$, which is $L_g$-isogenous to all of its Galois conjugates and has all its endomorphisms defined over $L_g$. We call $B$ the \emph{building block} of $A_g$.  

We say that a newform $f\in S_2(\Gamma_0(N))$ has \emph{inner twists} when there exists a non-trivial Dirichlet character $\chi$ such that $f$ and $f\otimes \chi$ are Galois conjugates. This is the case when $[E_f:F_f]>1$.  
\begin{proposition} \label{kolyvaginprop}
Suppose that $f\in S_2(\Gamma_0(N))$ is a newform without CM and without inner twists, $K$ is an abelian number field, and for each Dirichlet character $\chi: \mathrm{Gal}(K/\Q)\to \overline{\Q}^{\times}$ we have $L(f\otimes \chi,1)\neq 0$. Then for each such $\chi$, we have
\[
\mathrm{rk}(A_{f\otimes \chi}(K)) =0.
\]
\end{proposition}
\begin{proof}
Define $G=\mathrm{Gal}(K/\Q)$ and let $\widehat{G}$ be its character group. Consider $\chi\in \widehat{G}$ and define $g=f\otimes \chi$. 
By definition of the Weil restriction, we have
\[
A_{g}(K)=(\mathrm{Res}_{K/\Q}((A_{g})_K))(\Q),
\]
so we study the latter. Now $f$ having no inner twists means that $E_f=F_f$. Also note that $F_g=F_f$, because $g$ has Nebentypus character $\chi^2$. Recall that for each $s\in \mathrm{Gal}(E_g/F_g)$, we have a unique Dirichlet character $\chi_s$ such that $(f\otimes\chi)^s=f\otimes \chi\chi_s$. Indeed, in this case $\chi_s=\chi^s\chi^{-1}$, since $s$ leaves $f$ fixed. So $\chi_s\in \widehat{G}$, from which it follows that $L_g\subset K$. In \cite[Proposition 2]{guitart}, it is proved that $\mathrm{Res}_{L_g/\Q}((A_g)_{L_g})$ is $\Q$-isogenous to a product of modular abelian varieties $A_{f_i}$.  We explain how their proof extends to abelian number fields $K/L_g$.

Guitart and Quer showed in \cite[Theorem 5.3]{quer} the following: if $L$ is a Galois number field and $A/L$ an abelian variety, then 
\[
\mathrm{Res}_{L/\Q}(A)\sim_{\Q} A_{f_1}\times\ldots\times A_{f_n}
\]
for newforms $f_1,\ldots,f_n$ if and only if $L$ is abelian, $A$ is an \emph{$L$-building block} (defined in \cite[Definition 4.1]{quer}) and the cocycle class $[c_{A/L}]\in H^2(\mathrm{Gal}(L/\Q),Z(\mathrm{End}_L(A)\otimes \Q))$ (defined on \cite[p.\ 181]{quer}) is symmetric. Now let $B$ be the building block of $A_g$. Then $B$ is indeed an $L_g$-building block in the definition of Guitart and Quer, and hence a $K$-building block. We shall not define $c_{A/L}$, but note that it is a map $\mathrm{Gal}(L/\Q)\times \mathrm{Gal}(L/\Q)\to Z(\mathrm{End}_L(A)\otimes \Q)^{\times}$. In their proof of \cite[Proposition 2]{guitart}, Guitart and Gonz\'alez-Gim\'enez show that $c_{B/L_g}$ is symmetric. Now note that $Z(\mathrm{End}_K(B))=Z(\mathrm{End}_{L_g}(B))$ as all endomorphisms of $B$ are defined over $L_g$. It follows that $c_{B/K}$ is simply the composition of $\mathrm{Gal}(K/\Q)\to\mathrm{Gal}(L_g/\Q)$ and $c_{B/L_g}$. In particular, $c_{B/K}$ is also symmetric. We conclude that
\begin{align}\label{eqn100}
\mathrm{Res}_{K/\Q}(B)\sim_{\Q} A_{f_1}\times\ldots\times A_{f_n},
\end{align}
where $f_1,\ldots,f_n$ are newforms in some $S_2(\Gamma_1(M))$. Since $B$ is $K$-isogenous to all its Galois conjugates, it follows that $\mathrm{Res}_{K/\Q}(B)\sim_K B^m$ for some $m>0$. In particular, for each $i$ we must have $A_{f_i}\sim_K B^{m_i}$ for some $m_i>0$. Now, as noticed in \cite[Proposition 1]{guitart}, it follows from Ribet's work that if $A_{f_i}$ and $A_g$ are both $K$-isogenous to a power of $B$, then $g=f_i\otimes \psi$ for some Dirichlet character $\psi$ on $\mathrm{Gal}(K/\Q)$. We thus find that each $f_i=f\otimes \chi_i$ for some $\chi_i\in \widehat{G}$. From the isogeny $A_g\sim_K B^m$ for some $m>0$, we find that $\mathrm{Res}_{K/\Q}(A_g)_K$ is $\Q$-isogenous to $\mathrm{Res}_{K/\Q}(B)^m$, which by (\ref{eqn100}) is $\Q$-isogenous to a product of abelian varieties $A_{f\otimes \chi_i}$. As $L(f\otimes \chi_i,1)\neq 0$ by assumption, it follows from Kolyvagin--Logach\"ev that $\mathrm{rk}(A_{f\otimes \chi_i})(\Q)=0$ for each $i$, so that $\mathrm{rk}(A_g(K))=\mathrm{rk}(\mathrm{Res}_{K/\Q}(A_g)_K(\Q))=0$, as desired.
%
\end{proof}

\subsubsection{Generators of the free part}
Let $X/\Q$ be a curve. Quotients of the Jacobian $J(X)$ of positive Mordell--Weil rank can sometimes be identified as Jacobians of quotients of $X$, as is the case for $X_1,X_2,X(\mathrm{b}3,\mathrm{ns}7)$ and $X(\mathrm{b}5,\mathrm{ns}7)$. When such a quotient curve $C$ is elliptic, its rank can be determined by Cremona's algorithm \cite{cremona} implemented in \texttt{Magma}. When it is hyperelliptic of genus 2, as is the case for the aforementioned examples, Stoll's algorithm \cite{stoll} can be used to find generators. Those generators can be pulled back to $X$ and in favourable circumstances generate the free part of its Mordell--Weil group up to known index, as shown in \cite[Proposition 3.1]{box}:
\begin{proposition}\label{boxprop}Let $\rho\colon X\to C$ be a map of curves over $\Q$. If $\mathrm{rk}(J(X)(\Q))=\mathrm{rk}(J(C)(\Q))$, then 
\[
\mathrm{deg}(\rho)\cdot J(X)(\Q)\subset \rho^*J(C)(\Q)
\]
up to torsion.
\end{proposition}
We apply this proposition to find generators of subgroups of the Mordell--Weil groups of $X(\mathrm{b}5,\mathrm{ns}7)$ (see Section \ref{Xb5ns7mwgp}) and $X(\mathrm{b}3,\mathrm{ns}7)$ (see Remark \ref{Xb3ns7remark}).

\subsubsection{The torsion subgroup} For more details, we recommend the article of \"Ozman and Siksek \cite{ozman}, where many torsion subgroups of modular curves are computed using a variety of techniques. We found those methods to also be effective for our curves; we summarise them here.

Let $X$ be a modular curve, and denote by $C(\overline{\Q})\subset J(X)(\overline{\Q})$ the subgroup generated by the differences of cusps. The \emph{cuspidal subgroup} $C(\Q)$ of $J(X)(\Q)$ is the group of $\mathrm{Gal}(\overline{\Q}/\Q)$-invariants of $C(\overline{\Q})$. By the Manin--Drinfeld theorem \cite{manin}, \cite{drinfeld}, $C(\Q)\subset J(X)(\Q)_{\mathrm{tors}}$ when $X=X_0(N)$ for any $N$. Mazur \cite{mazur1} proved that in fact $C(\Q)=J(X_0(N))(\Q)_{\mathrm{tors}}$ when $N$ is prime, which was a conjecture of Ogg. The same statement for any $N$ is often called Ogg's Conjecture \cite{ogg}. See e.g.\ \cite{oggconj2} for partial results, and \cite[Section 5.2]{ozman} for more details.

Following this idea, the cuspidal subgroup provides a good starting point for attempting to compute $J(X)(\Q)_{\mathrm{tors}}$ even when $X$ is not of the form $X_0(N)$. For example, on $X(\mathrm{b}5,\mathrm{ns}7)$, there are 6 cusps defined over $\Q(\zeta_7+\zeta_7^{-1})$ splitting into two irreducible $\Q$-rational degree 3 divisors $c_0,c_{\infty}$, and their difference $[c_0-c_{\infty}]$ generates $J(\mathrm{b}5,\mathrm{ns}7)(\Q)_{\mathrm{tors}}\simeq \Z/7\Z$, as was shown in \cite{derickx}. Moreover, we show in Propositions \ref{X0105w5mwgp} and \ref{mwgpsprop} that the rank 0 Mordell--Weil groups of $X_0(105)/w_5$ and $X_0(105)/w_{35}$ are generated by differences of images of cusps. 
\begin{remark}
Ogg's conjecture cannot simply be extended to quotients of $X_0(N)$. For example, when $p$ is prime,  $X_0(p)$ has two cusps $c_1$ and $c_2$ satisfying $c_2=w_p(c_1)$. On $X_0(p)/w_p$, the cuspidal subgroup is therefore trivial, whereas there can be torsion. For $X_0(105)/\langle w_7,w_{105}\rangle$ and $X_0(105)/\langle w_7,w_{21}\rangle$, we found in Proposition \ref{mwgpsprop} the torsion subgroup to be larger than the cuspidal subgroup.
\end{remark}

Given a subgroup $G\subset J(X)(\Q)_{\mathrm{tors}}$, it can often be verified that $G=J(X)(\Q)_{\mathrm{tors}}$ by using the fact (see \cite{katz}) that $J(X)(\Q)_{\mathrm{tors}}$ embeds into $J(X)(\F_p)$ for good primes $p>2$. The Mordell--Weil group over $\F_p$ can be determined using an algorithm of Hess \cite{hess} (implemented in \texttt{Magma}), and a combination of a few primes often leaves no other option but $G=J(X)(\Q)_{\mathrm{tors}}$. This is indeed the case for $X_0(35)$ and $X(\mathrm{s}3,\mathrm{b}5)$; see Proposition \ref{mwgrpsprop}.

However, for some curves (such as $X_0(105)/w_5$, $X_0(105)/w_{35}$, $X_0(105)/\langle w_7,w_{105}\rangle$, $X_0(105)/\langle w_7,w_{21}\rangle$ and $X(\mathrm{b}3,\mathrm{ns}7)$), this strategy fails to determine $J(X)(\Q)_{\mathrm{tors}}$, because there exists an abstract group $H$ such that $G\subsetneq H$ and $H\subset J(X)(\F_p)$ (as abstract groups) for all primes $p$ considered. In practice, often $H$ has larger 2-torsion than $G$. This may be caused by the existence of fields $K_1,\ldots,K_n$ such that each prime $p$ splits in one of them and each $J(X)(K_i)_{\mathrm{tors}}$ has an extra 2-torsion element that is not $\Q$-rational. 

This can often be remedied:
\begin{itemize}
    \item When $X$ is a hyperelliptic curve of genus 2 (e.g.\ $X=X_0(105)/\langle w_7,w_{105}\rangle$, $X_0(105)/\langle w_7,w_{21}\rangle$ or $X(\mathrm{b}3,\mathrm{ns}7)/w_3$), algorithms for computing $J(X)(\Q)[2]$ are implemented in \texttt{Magma}.
    \item When $X$ is a plane quartic (e.g.\ $X=X_0(105)/w_{35}$), the Galois representation $\mathrm{Gal}(\overline{\Q}/\Q)\to \mathrm{Aut}_{\overline{\Q}}J(X)[2]$ can be determined by computing bitangents using an algorithm of Bruin, Poonen and Stoll \cite[Section 12]{bps}; see Proposition \ref{mwgpsprop} for details.
    \item Sometimes (e.g.\ for $X=X_0(64)$  (see \cite{ozman}) or $X=X_0(105)/w_5$) for an extension $K/\Q$ an extra 2-torsion element can be found, making it possible to determine $J(X)(K)[2]$ exactly. Then $J(X)(\Q)[2]$ is its subset of  Galois invariants. 
    \item While $H$ may as an abstract group be isomorphic to subgroups $H_p\subset J(X)(\F_p)$ for each $p$, there may not exist isomorphisms $\psi_{p,q}:\; H_p\simeq H_q$ for all such primes $p$ and $q$ such that $\psi_{p,q}$ restricted to the image of $G$ is the map obtained from reducing $G$ modulo $p$ and $q$. This can be verified using an algorithm of \"Ozman and Siksek \cite{ozman}. We apply this to $X_0(105)/w_5$ in Proposition \ref{X0105w5mwgp}. 
\end{itemize}

\subsection{When the Mordell--Weil group is finite} \label{finitemwgpsec}
 When the Mordell--Weil group of a curve $X$ over a field $K$ is finite and known,  it can be used (for suitably small $d$) to determine $X^{(d)}(K)$ using the Abel--Jacobi map $\iota_{D_0}$ defined in Section \ref{sympowsec}. When $D$ is a divisor on $X$, we define
\[
L(D)\colonequals \{f\in K(X)^{\times} \mid \mathrm{div}(f)+D\geq 0\}\cup \{0\}.
\]
This is a finite-dimensional vector space called the \emph{Riemann--Roch space} of $D$, and an efficient implementation for finding bases of such spaces is available in \texttt{Magma}. We denote its dimension by $\ell(D)$. We obtain the following tautological lemma.
\begin{lemma}\label{lem2.3}
For each degree 0 divisor $E$ on $X$, we have
\[
\{D\in X^{(d)}(K) \mid [D-D_0]=[E]\}=\{E+D_0+\mathrm{div}(f) \mid f\in \P L(E+D_0)\}.
\]
\end{lemma}
Ranging over representatives $E$ of the elements in $J(X)(K)$, this yields an algorithm for determining $X^{(d)}(K)$ when it is finite. (This may even work when only a finite index subgroup of $J(X)(K)$ is known, in which case computations can be sped up with a Mordell--Weil sieve, as explained in \cite{ozman}.)  Even when $X^{(d)}(K)$ is infinite, it can be useful to identify the Riemann--Roch spaces of dimension $>1$. 
\begin{corollary}\label{rrcor}
Suppose that $\rho:X\to C$ is a degree $d$ morphism of curves over $K$, and $D_0$ is $K$-rational divisor of degree $e$ on $C$. If $L(\rho^*D_0)=\rho^*L(D_0)$, we have
\[
\{D\in X^{(de)}(K)\mid [D-\rho^*D_0]=0\}\subset \rho^*C^{(e)}(K).
\]
\end{corollary}
\begin{proof}
This follows from Lemma \ref{lem2.3} with $E=0$, using the fact that $\mathrm{div}(\rho^*f)=\rho^*\mathrm{div}(f)$.
\end{proof}
We use this argument repeatedly in Sections \ref{X0105sec} and \ref{Xs3b5b7sec}.

\subsection{The Mordell--Weil sieve}
\label{mwsieve}
For both finite and infinite Mordell--Weil groups, the Mordell--Weil sieve can be an effective tool to determine the ratinal points on symmetric powers of curves. The applied Mordell--Weil sieve is similar to those used in \cite{siksek}, \cite{box} and \cite{bgg}, but we describe it in a slightly more general form, allowing us to use it in Section \ref{Xb5ns7sec} as well as Section \ref{X1andX2sec}. For a more detailed discussion of the Mordell--Weil sieve, we refer to \cite{bruinstoll}. The idea of the sieve is to combine incomplete information about rational points in $p$-adic discs for various primes $p$, in order to determine the set of all rational points.

Consider a variety $V/\Q$, an integer $N$ and a proper Noetherian model $\mathcal{V}/\Z[1/N]$ for $V$. Then we have reduction maps $V(\Q)\to \mathcal{V}(\F_p)$ for primes $p\nmid N$.
\begin{remark}
Typically, we consider a curve $X/\Q$ and $m\in \Z_{\geq 1}$, and define $V=X^{(m)}$. Then, given equations with integral coefficients for $X$ in projective space, define $N$ to be the product of the primes of bad reduction for this model of $X$, let $\mathcal{X}/\Z[1/N]$ be the curve defined by those equations, and set $\mathcal{V}=\mathcal{X}^{(e)}$. 
\end{remark}

We first choose primes $p_1,\ldots,p_r$ not dividing $N$, and we consider a subset $S\subset V(\Q)$. Next, we need the following input:
\begin{itemize}
    \item[(i)] A (possibly infinite) excplicit list of \emph{known points} $\mathcal{L}\subset S$.
    \item[(ii)] A $\Z[1/N]$-morphism $\iota: \mathcal{V}\to \mathcal{A}$, where $\mathcal{A}/\Z[1/N]$ is an abelian scheme. 
    \item[(iii)] A finitely generated abelian group $B$, a homomorphism $\phi: B\to \mathcal{A}(\Q)$ and a subset $W\subset B$ such that $\iota(S)\subset \phi(W)$. 
    \item[(iv)] For each $p\in \{p_1,\ldots,p_n\}$, a subset $\mathcal{M}_p\subset \mathcal{V}(\F_p)$ such that $S\setminus \mathcal{L}$ reduces into $\mathcal{M}_p$.
\end{itemize}
The aim of the sieve is to show that $\mathcal{L}=S$. Typically, $S=V(\Q)$ and $W=B$, but the flexibility in this set-up will prove useful.
\begin{remark}
Note that $\mathcal{M}_p$ need not be optimal. Strict subsets $\mathcal{M}_p$ can be obtained from theorems such as Theorem \ref{chabthm}, using points $\mathcal{Q}$ in an explicit finite subset $\mathcal{L}'\subset \mathcal{L}$.
\end{remark}
\begin{remark}\label{mwrk}
When $V=X^{(e)}$ for a curve $X$,  the abelian scheme $\mathcal{A}$ is usually derived from the Jacobian of $X$. For example, it could be the image of the multiplication-by-$I$ map $m_I$ on the Jacobian, and $\iota$ could be composition of the Abel--Jacobi map with $m_I$. The choice of $\mathcal{A}$ is then limited by knowledge of generators of $J(X)(\Q)$.
\end{remark}
Consider $p\in \{p_1,\ldots,p_n\}$. We obtain a commuting diagram
\[
\begin{tikzcd}
\mathcal{L} \arrow[rr] \arrow[rrd] &  & V(\Q) \arrow[d] \arrow[rr, "\iota"] &  & \mathcal{A}(\Q) \arrow[d] &  & B \arrow[ll, "\phi"'] \arrow[lld, "\phi_p"] \\
                                   &  & \mathcal{V}(\F_p) \arrow[rr, "\iota_p"]                                 &  & \mathcal{A}(\F_p)                                  &  &                                            
\end{tikzcd},
\]
from which the following proposition follows by definition of $\mathcal{M}_p$. 
\begin{proposition}[Mordell--Weil sieve]\label{mwsieveprop}If
\[
W\cap \bigcap_{i=1}^n\phi_{p_i}^{-1}(\iota_{p_i}(\mathcal{M}_{p_i}))=\emptyset
\]
then $S=\mathcal{L}$.
\end{proposition}
\begin{remark}\label{sieveproportion}
One may hope for this sieve to work in practice when $\mathrm{dim}(\mathcal{A_{\overline{\Q}}})>\mathrm{dim}(V)$, as in each step, less than a proportion
\[
\#\mathcal{V}(\F_p)/\#\mathcal{A}(\F_p)\sim p^{\mathrm{dim}(V)-\mathrm{dim}(\mathcal{A}_{\overline{\Q}})}
\]
of elements of $\mathcal{A}(\F_p)$ is in $\iota(\mathcal{M}_p)$. 
\end{remark}
When $X$ is a curve with good reduction $\widetilde{X}$ at $p$, we note that $J(\widetilde{X})(\F_p)$ can be computed using the class group algorithm of Hess \cite{hess}. We have implemented this sieve with the practical improvements described in \cite[Section 3]{bgg}. 
\subsection{When the Mordell--Weil group is infinite: Chabauty}\label{infmwgpsec}
The Mordell--Weil sieve just described only succeeds in determining a non-empty set of rational points $V(\Q)$ when combined with a method to determine the rational points within a fixed mod $p$ residue class. In other words, we require a way to determine strict subsets $\mathcal{M}_p$. 

Traditionally, when $V$ is a curve $X$, this has been done using the Chabauty--Coleman method, provided the rank $r$ of $J(X)(\Q)$ and the genus $g$ of $X$ satisfy $r<g$. This was generalised to rational points of symmetric powers $X^{(m)}$ of curves by Siksek \cite{siksek}, under the stronger (but not sufficient) condition $r<g-(m-1)$.

An interesting situation occurs when $X$ admits a morphism $\rho: X\to C$ to another curve. Suppose that $\rho$ has degree $d$, and we want to compute $X^{(m)}(K)$, where $K$ is a number field and $m\geq e\cdot d$ for some $e\in \Z_{\geq 1}$. A point $\mathcal{P}\in X^{(m-ed)}(K)$ (when $e=md$, we take $\mathcal{P}=0$) gives rise to a subset
\[
\mathcal{P}+\rho^*C^{(e)}(K)\subset X^{(m)}(K).
\]
When this subset is infinite, or is finite but cannot be determined provably, this is a clear obstruction to Siksek's symmetric Chabauty method. 
Siksek \cite{siksek} adapted his symmetric Chabauty method, allowing him to also compute the isolated points
\begin{align}\label{chabset}
X^{(m)}(K)\setminus \bigcup_{\substack{e\text{ s.t. } ed\leq m\\ \mathcal{P}\in X^{(m-ed)}(K)}}\left(\mathcal{P}+\rho^*C^{(e)}(K)\right),
\end{align}
in the special case where $m=d$ and $\mathcal{P}=0$. When $m=2$ (quadratic points), all such subsets are of the form $\rho^*C(K)$ where $\rho$ is a morphism of degree 2, but for higher degrees, subsets of the form $\mathcal{P}+\rho^*C^{(e)}(K)$ with $\mathcal{P}\neq 0$ occur. The author, Gajovi\'c and Goodman \cite{bgg} recently generalised Siksek's method to arbitrary values of $m,e$ and $d$ satisfying $ed\leq m$, opening the way to applying Chabauty to symmetric powers $X^{(d)}$ for $d>2$. We give a brief description of this new method. \\

Consider a prime $\pp$ of $K$, and suppose that $\mathcal{X}/\O_{K_{\pp}}$ and $\mathcal{C}/\O_{K_{\pp}}$ are minimal proper regular models for $X/K_{\pp}$ and $C/K_{\pp}$ respectively, such that $\rho$ extends to an $\O_{K_{\pp}}$-morphism $\rho: \mathcal{X}\to \mathcal{C}$. (In particular, we assume that $X$ and $C$ have good reduction at $\pp$.) We obtain a commutative diagram
\[
\begin{tikzcd}
X^{(m)}(K) \arrow[rr, "\iota_X"] \arrow[d] &  & J(X)(K) \arrow[d] \\
X^{(m)}(K_{\pp}) \arrow[rr, "\iota_C"]         &  & J(X)(K_{\pp}).       
\end{tikzcd}
\]
The idea (due in the $m=1$ case to Chabauty \cite{chabauty}) is to consider $X^{(m)}(K_{\pp})\cap J(X)(K)$, which contains $X^{(m)}(K)$. Coleman \cite{coleman} developed (in the $m=1$ case) an effective method to bound this set. 

Let $A/K_{\pp}$ be an abelian variety. Coleman defined in \cite[Section II]{coleman2} a pairing, now called \emph{Coleman integration},
\begin{align}
\label{colemanintegration}
A(K_{\pp})\times H^0(A,\Omega)\longrightarrow K_{\pp},\;\; (D,\om)\mapsto \int_D\om
\end{align}
which is $\Z$-linear on the left, $K_{\pp}$-linear on the right and locally analytic in $D\in A(K_{\pp})$, with left-hand kernel equal to $A(K_{\pp})_{\mathrm{tors}}$ and trivial right-hand kernel. Here $\Omega$ denotes the sheaf of regular differential forms of the first kind. Moreover, given a morphism $\pi\colon A\to B$ of abelian varieties over $K_{\pp}$, the pairing satisfies a chain rule
\begin{align}\label{chainrule}
\int_{D}\pi^*\om = \int_{\pi_*D}\om \text{ for all } D\in A(K_{\pp})\text{ and }\om \in H^1(A,\Omega).
\end{align}
The Abel--Jacobi map $\iota_X\colon X\to J(X)$ yields an isomorphism $\iota_X^*\colon H^0(J(X),\Omega_{J(X)})\simeq H^0(X,\Omega_X)$. The pairing (\ref{colemanintegration}) thus specialises to a pairing $J(X)(K)\times H^0(X_{K_{\pp}},\Omega)\to K_{\pp}$, which we also denote by $(D,\om)\mapsto \int_D\om$.
\begin{definition}\label{vandiffdef}
We define the \emph{space of vanishing differentials} $V\subset H^0(X_{K_{\pp}},\Omega)$ to be the annihilator of $J(X)(K)$ under the pairing just defined, and the \emph{space of trace zero vanishing differentials} $V_C$ to be the intersection of $V$ and the kernel of $\mathrm{Tr}:\; H^0(X_{K_{\pp}},\Omega^1)\to H^0(C_{K_{\pp}},\Omega^1)$.
\end{definition}
By the properties of (\ref{colemanintegration}) and linear algebra, we find
\[
\mathrm{dim}(V)\geq g_X-r_X \text{ and } \mathrm{dim}(V_C)\geq g_X-g_C - (r_X-r_C).
\]

Now each $D\in X^{(m)}(K)$ satisfies
\[
\int_{\iota_X(D)}\om = 0 \text{ for all } \om \in V,
\]
and the aim is thus to compute the common zero set of these integrals on $X^{(m)}(K_{\pp})$. The $\om \in V_C$ moreover are ``indifferent'' to $\rho^*J(C)$, in the sense that $\int_{\rho^*D}\om=0$ for all $D\in J(C)$ and $\om \in V_C$. 

This integral breaks down into a sum of integrals of the form $\int_{[R-Q]}\om$ for points $R,Q\in X$ defined over an extension of $K$ of degree at most $m$. Consider $\omega\in H^0(\mathcal{X},\Omega_{\mathcal{X}/\O_{K_{\pp}}})$, a field extension $L/K$, a point $Q\in X(L)$ and denote by $\widetilde{Q}$ the reduction of $Q$ modulo a prime $\qq$ of $L$ above $\pp$. Consider a local coordinate $t\in \O_{X,Q}$ reducing to a local coordinate at $\widetilde{Q}$.  Suppose the expansion of $\om$ around $Q$ is
\[
\om=(a_0(\om,t)+a_1(\om,t)t+a_2(\om,t)t^2+\ldots)\dd t \in \widehat{\Omega}_Q,
\]
where $\widehat{\Omega}_Q$ is the completion of the stalk of $\Omega_{X/L_{\qq}}$ at $Q$. Then each $a_i(\om,t)\in \O_{L_{\qq}}$. If $R\in X(L)$ is in the \emph{residue class} of $Q$, i.e.\ reduces to the same point as $Q$ modulo $\qq$, then we can evaluate their Coleman integral (see \cite{coleman2}), which we call a \emph{tiny integral}:
\[
\int_{[R-Q]}\om = \sum_{n=1}^{\infty}\frac{a_{n-1}}{n}t(R)^n.
\]
When $L_{\qq}=\Q_p$ (corresponding to $m=1$), integrals between points in \emph{different} residue classes can also be determined, thanks to the work of Tuitman and Balakrishnan \cite{tuitman}, who also wrote a \texttt{Magma} implementation. This has not (yet) been extended to $[L_{\qq}:\Q_p]>1$, however.

Instead, we therefore stick to integrals within residue classes; this suffices when combining information obtained by using multiple primes $\pp$, using the Mordell--Weil sieve. Thanks to the sieve, it suffices to determine the intersection of $X^{(m)}(K)$ with the residue classes $D(\mathcal{Q})\subset X^{(m)}(\overline{K}_{\pp})$ of known points $\mathcal{Q}\in X^{(m)}(K)$. A criterion `\`a la Siksek' to decide when $X^{(m)}(K)\cap D(\mathcal{Q})$ is as small as it can be, was found in \cite{bgg} by studying the common zero sets of the power series obtained from tiny integrals. In Sections \ref{deg4chabsec} and \ref{deg2chabsec}, we describe it concretely when $K=\Q$ and $m=4$ and $m=2$, respectively. \\

In Section \ref{Xb5ns7sec}, we apply this method to the genus 6 curve $X(\mathrm{b}5,\mathrm{ns}7)$, which admits a degree 2 map $\rho$ to the genus 2 curve $C\colonequals X(\mathrm{b}5,\mathrm{ns}7)/w_5$, and contains two effective degree 2 divisors $D_1$ and $D_2$. To be precise, the set (\ref{chabset}) we determine (in conjunction with the Mordell--Weil sieve) is
\[
X(\mathrm{b}5,\mathrm{ns}7)^{(4)}(\Q)\setminus \left(\rho^*C^{(2)}(\Q)\cup(D_1+\rho^*C(\Q))\cup(D_2+\rho^*C(\Q))\right).
\]
In Section \ref{X1andX2sec}, we use Siksek's symmetric Chabauty method (and a sieve) to find the degree 2 effective divisors not mapping to $C$ on $X(\mathrm{b}5,\mathrm{e}7)/w_5$ (genus 5) and another curve $X_1$ (genus 8).

\section{$X_0(105)$}\label{X0105sec}
In this section we prove the following theorem.
\begin{theorem}
\label{X0105thm}
Let $E$ be an elliptic curve with quartic $j$-invariant that occurs as a quartic point on $X_0(105)$. Then $E$ is a $\Q$-curve, and in particular $E$ is modular. 
\end{theorem}
\subsection{A convenient model}
As described in Section \ref{ranksec}, we computed that $L(f,1)\neq 0$ for each eigenform $f\in S_2(\Gamma_0(N))$, from which we conclude by \cite{kolyvagin} that the Mordell--Weil group $J_0(105)(\Q)$ is finite. The genus of $X_0(105)$ is 13, however, making it computationally infeasible to do the Riemann--Roch space computations described in Section \ref{finitemwgpsec}. Instead, we begin by studying the quotient $X_0(105)/w_5$, where $w_5$ is the Atkin--Lehner involution corresponding to 5. This quotient has genus 5. 

We compute the map $X_0(105)\to X_0(105)/w_5$ explicitly. To this end, we first find a basis of 13 linearly independent cusp forms in $S_2(\Gamma_0(105))$, consisting of 5 cusp forms that are fixed by $w_5$ and 8 cusp forms $f$ such that $f|w_5=-f$. Equations between the first five cusp forms yield the following model for $X_0(105)/w_5$ in $\P^4_{x_1,\ldots,x_5}$:
\begin{align*}
X_0(105)/w_5: \;\;\;& x_1x_3 - x_2^2 - x_3x_5 - x_4^2=0,\\
&x_1x_4 - x_2x_3 + x_2x_4 - x_3^2 - x_4x_5=0,\\
&2x_1x_5 +  x_2x_4 - 2x_2x_5 - 3x_3^2 + 3x_3x_4 + 2x_3x_5 - 2x_4^2 - 2x_4x_5 + 2x_5^2=0,
\end{align*}
and the (quadratic) equations between all 13 cusp forms result in an explicit canonical model for $X_0(105)\subset \P^{12}$. We do not display this model here (it consists of 55 quadratic equations in 13 variables), but we invite the curious reader to run the \texttt{Magma} code. The quotient map $X_0(105)\to X_0(105)/w_5$ is then simply projection onto the first 5 coordinates. 

Since the map $X_0(105)\to X_0(105)/w_5$ is of degree 2, there must be a $g\in \Q(X_0(105)/w_5)$, not a square, whose pullback to $\Q(X_0(105))$ is a square. In fact, we find such a $g$ in the coordinate ring of our model for $X_0(105)/w_5$. This reveals to us the extra equation $z^2=g$, where $z$ is a new coordinate, which determines $X_0(105)$ as a degree 2 cover of $X_0(105)/w_5$. Using the equations defining $X_0(105)$, we search for a polynomial $f\in \Q[x_1,\ldots,x_{13}]\setminus \Q[x_1,\ldots,x_5]$ such that $f^2$ is, modulo the equations defining $X_0(105)$, a polynomial in $x_1,\ldots,x_5$. We square $f$ to obtain
\[
g= 9x_1^2 - 6x_1x_2 + x_2^2 + 6x_2x_3 + 21x_2x_4 - 
    12x_2x_5 - 17x_3^2 + 19x_3x_4 + 18x_3x_5 + 7x_4^2 - 
    18x_4x_5 + 27x_5^2.
\]
It can be easily verified that $g$ is indeed a square in the coordinate ring of $X_0(105)$, but that $g/x_1^2$ is not a square in $\Q(X_0(105)/w_5)$. We conclude that
$
X_0(105)\to X_0(105)/w_5
$
is the degree 2 cover given by $z^2=g$. We denote the coordinates of $X_0(105)\subset \P^5$ in this model by $(x_1:\ldots : x_5 :z )$. We obtain the following lemma, which proves the convenience of this model. 

\begin{lemma}
\label{squarechecklemma}
Suppose that $K$ is a field and $P\in (X_0(105)/w_5)(K)$. Let $Q\in X_0(105)$ be a point mapping to $P$. Then $Q\in X_0(105)(K)$ if and only if $g(P)$ is a square in $K$. 
\end{lemma}

\subsection{Quartic points on $X_0(105)/w_5$}
Next, we study the quartic points on $X_0(105)/w_5$. Its cover $X_0(105)$ has eight cusps, mapping to the four rational points 
\begin{align*}
    P_1=(3 : 0 : 2 : 2 : 1), \;P_2=(0 : 0 : -1 : -1 : 1), \; P_3=(1 : 0 : 0 : 0 : 0), \;P_4=(-1 : 0 : 0 : 0 : 1)
\end{align*}
on $X_0(105)/w_5$.
Our curve $X_0(105)/w_5$ has three quotients by Atkin--Lehner involutions: \[
\pi_3: X_0(105)/w_5\to C\colonequals X_0(105)/\langle w_5,w_3\rangle, 
\]
\[\pi_7: X_0(105)/w_5\to E_7\colonequals X_0(105)/\langle w_5,w_7\rangle
\]
and 
\[
\pi_{21}: X_0(105)\to E_{21}\colonequals X_0(105)/\langle w_5,w_{21}\rangle.
\]
Here $C$ is hyperelliptic of genus 3. The curves $E_7$ and $E_{21}$ are elliptic curves of conductors 21 and 35 respectively. On all three of these, the remaining Atkin--Lehner operator acts as a non-trivial involution. Here are their models:
\begin{align*}
C:&\;\; y^2 = x^8 + 6x^7 + 7x^6 - 2x^5 + 17x^4 + 8x^3 - 28x^2 + 36x - 24,\\
E_7:&\;\; y^2 + xy = x^3 + x \text{ and }\\
E_{21}:&\;\; y^2 + y = x^3 + x^2 - x.
\end{align*}

\begin{proposition}\label{X0105w5mwgp}
The Mordell--Weil group of $X_0(105)/w_5$ equals
\[
\Z/6\Z \cdot (7[P_2-P_1]-6[P_3-P_1])\oplus \Z/24\Z\cdot [P_4-P_1].
\]
\end{proposition}
\begin{proof}
We denote our curve by $X\colonequals X_0(105)/w_5$ and its Jacobian by $J$. We know that $J(\Q)$ has rank 0, as it maps to $J_0(105)(\Q)$ with finite kernel. In particular, at each prime $p>2$ of good reduction for the displayed model, $J(\Q)$ embeds into $J(\F_{p})$. The primes $p=11$ and $p=13$ are of good reduction. Using the embedding for $p=11$, we check that the cuspidal subgroup of $J(\Q)$ generated by $[P_4-P_1],[P_3-P_1],[P_2-P_1]$ equals the claimed group. We use the class group algorithm of Hess \cite{hess}, implemented in \texttt{Magma}, to find that $J(\F_{11})\simeq \Z/2\Z\times \Z/2\Z\times \Z/60\Z \times \Z/600\Z$ and $J(\F_{13})\simeq \Z/2\Z \times \Z/6\Z \times \Z/6\Z \times \Z/3168\Z$. As $\gcd(600,3168)=24$, we find that $J(\Q)$ is isomorphic to one of
\[
\Z/6\Z\times \Z/24\Z,\quad \Z/2\Z\times \Z/6\Z \times \Z/24\Z,\quad \Z/2\Z^2\times \Z/6\Z\times \Z/24\Z.
\]
It thus suffices to determine the 2-torsion of $J(\Q)$. Considering more primes does not narrow it down any further. Instead, we will show that 
\[
J(\Q(\sqrt{21}))[2]\simeq (\Z/2\Z)^3
\]
and determine $J(\Q)[2]$ by taking Galois invariants. In order to find the extra 2-torsion element, we look at the quotient $C=X/w_3$. The quadratic point
$
(2,4\sqrt{105})\in C
$ 
generates a rational effective degree 2 divisor pulling back to a degree 4 divisor $D$ on $X$, irreducible over $\Q$, whose points are defined over $\Q(\sqrt{5},\sqrt{21})$. Over $\Q(\sqrt{21})$, however, $D$ splits into a sum of two effective degree 2 divisors: $D=D_1+D_2$, where $D_1$ is the sum of
\begin{align*}
(&(-254\sqrt{21} + 1326)\sqrt{5} + 5(-111\sqrt{21} + 594):
    (84\sqrt{21} - 516)\sqrt{5} + 5(38\sqrt{21} - 210):\\
    &\;\;\;\;\;5(-4\sqrt{21} + 48)\sqrt{5} + 5(-11\sqrt{21} + 91):
    (22\sqrt{21} - 18)\sqrt{5} + 5(8\sqrt{21} - 14):
    205)
\end{align*}
and its $\mathrm{Gal}(\Q(\sqrt{5},\sqrt{21})/\Q(\sqrt{21}))$-conjugate. We then verify that 
\[
D_{\mathrm{tors}}\colonequals [D_1-2P_1]-[P_2-P_1]-2[P_3-P_1]-[P_4-P_1]
\]
is a non-trivial two-torsion element in $J(\Q(\sqrt{21}))$, and moreover that the difference of $D_1$ and its $\mathrm{Gal}(\Q(\sqrt{21})/\Q)$-conjugate is not principal. This means that $D_{\mathrm{tors}}$ is not defined over $\Q$. In particular, $J(\Q(\sqrt{21}))[2]$ must have dimension at least 3 over $\F_2$. 

Next, note  that 17 and 47 split in $\Q(\sqrt{21})$, so that $J(\Q(\sqrt{21}))_{\mathrm{tors}}$ embeds into $J(\F_{17})$ and $J(\F_{47})$. Both of these groups have 2-torsion isomorphic to $(\Z/2\Z)^4$. Let $H\colonequals \langle[P_2-P_1],[P_3-P_1],[P_4-P_1]\rangle\subset J(\Q)$. We use the algorithm developed by \"Ozman and Siksek \cite{ozman} to show the following: there does not exist a group isomorphism $\phi: H_{17}\to H_{47}$, such that $H_{17}$ and $H_{47}$ are subgroups of $J(\F_{17})$ and $J(\F_{47})$ respectively, both containing the image of $H$ and the subgroup isomorphic to $(\Z/2\Z)^4$, where the restriction of $\phi$ to the image of $H$ is the isomorphism obtained by reducing $H$ modulo 17 and 47. It follows that $J(\Q(\sqrt{21}))[2]\simeq (\Z/2\Z)^3$ and $J(\Q)[2]=J(\Q(\sqrt{21}))[2]^{\mathrm{Gal(\Q(\sqrt{21})/\Q)}}\simeq (\Z/2\Z)^2$, as desired.
\end{proof}

The hyperelliptic curve $C$ admits the involution $w_7$, giving rise to a degree 2 map $C\to \P^1=X_0(105)/\langle w_3,w_5,w_7\rangle$. The curve $C$ thus has an infinite set of quadratic points coming from $\P^1(\Q)$ under this map. The two elliptic curves, however, have more sources of infinitely many quadratic points: by the Riemann--Roch theorem, there is an infinite set of effective degree 2 divisors mapping to each point in the Mordell--Weil groups $E_7(\Q)$ and $E_{21}(\Q)$ under the Abel--Jacobi map (c.f.\ Section \ref{finitemwgpsec}). All of these give rise to quartic points on $X_0(105)/w_5$. We compute that $E_7(\Q)\simeq \Z/4\Z$ and $E_{21}(\Q)\simeq \Z/3\Z$. 

Denote the generators of $J(X_0(105)/w_5)(\Q)$ by 
\[
D_1\colonequals 7(P_2-P_1)-6(P_3-P_1) \text{ and } D_2\colonequals P_4-P_1,
\]
 and write 
\[
D_0\colonequals P_1+P_2+P_3+P_4.
\]
The two elliptic curves $E_7$ and $E_{21}$ supply us with a total of 6 classes $[D]$ in $J(X_0(105)/w_5)(\Q)$ such that $\ell(D+D_0)\geq 2$, by pulling back the elements in the Mordell--Weil groups $E_7(\Q)$ and $E_{21}(\Q)$ (there is one overlapping class). As we know the Mordell--Weil group of $X_0(105)/w_5$, we can verify that pullbacks of quadratic points from $E_7$ and $E_{21}$ are the only sources of infinitely many quartic points on $X_0(105)/w_5$ by computing Riemann--Roch spaces, c.f.\ Section \ref{finitemwgpsec}.

\begin{proposition}
\label{finitesetprop}
Consider $(a,b)\in \Z/6\Z\times \Z/24\Z$. The dimension of the Riemann--Roch space
$
L(aD_1+bD_2+D_0)
$
is at least 2 precisely when $(a,b)\in \{(0,0),(0,6),(0,8),(0,12),(0,16),(0,18)\}$, in which case the dimension equals 2. As the above set has size 6, each of these Riemann--Roch spaces is the pullback of a corresponding Riemann--Roch space on $E_7$ or $E_{21}$ (or both, when $(a,b)=(0,0)$). In particular, the set
\[
(X_0(105)/w_5^{(4)})(\Q)\setminus \left(\pi_7^*E_7^{(2)}(\Q)\cup \pi_{21}^*E_{21}^{(2)}(\Q)\right)
\]
is finite and explicitly computable.
\end{proposition}
\begin{proof}
We apply Lemma \ref{lem2.3} to the Riemann--Roch spaces $L(D)$ of dimension 1 to compute the points. By Corollary \ref{rrcor}, the elements $D\in (X_0(105)/w_5^{(4)})(\Q)$ such that $\ell(D)=2$ correspond precisely to pullbacks because the Riemann--Roch spaces are pullbacks. The Riemann--Roch computations were done explicitly in \texttt{Magma}, which took multiple hours.
\end{proof}

\begin{proposition}
\label{modw5prop}
Each quartic point on $X_0(105)$ maps to a rational or quadratic point on $E_7$ or $E_{21}$. 
\end{proposition}
\begin{proof}
Each quartic point on $X_0(105)$ maps to a quartic or quadratic point on $X_0(105)/w_5$. Using Proposition \ref{finitesetprop}, we find the quartic points on $X_0(105)/w_5$ not mapping to a quadratic point on $E_7$ or $E_{21}$, and we find that none of those are the image of a quartic point on $X_0(105)$ by Lemma \ref{squarechecklemma}. There are three pairs of quadratic points on $X_0(105)/w_5$. We can find these either by decomposing the degree 4 divisors found in Proposition \ref{finitesetprop} into irreducible divisors, or by doing the equivalent Riemann--Roch space computation for degree 2 points. We verify that each of these three quadratic points maps to a rational point on $E_7$ or $E_{21}$. 
\end{proof}
\subsection{Exploiting the tree}

Next, we focus our attention to the quadratic points on $E_7$ and $E_{21}$. There are infinitely many of those, but few are in the image of a quartic point on $X_0(105)$. The following lemma allows us to exploit this.

\begin{lemma}
\label{diagramlemma}
Suppose that $X$ is a curve over a field $K$ admitting two distinct commuting involutions $v,w:\; X\to X$. Suppose that $Q\in X$ is defined over a quartic extension of $K$, but its image in $X/\langle v,w\rangle$ is defined over a quadratic extension of $K$. Then the image of $Q$ is quadratic over $K$ in one of the three curves
$
X/v$,  $X/w$ and $X/vw$.
\end{lemma}
\[
\begin{tikzcd}
               & X \arrow[ld] \arrow[d] \arrow[rd] &                \\
X/v \arrow[rd] & X/vw \arrow[d]                    & X/w \arrow[ld] \\
               & {X/\langle v,w\rangle}            &               
\end{tikzcd}
\]
\begin{proof}
Consider the effective degree 2 divisor $D\in (X/\langle v,w\rangle)^{(2)}(K)$ that is the sum of the image of $Q$ and its conjugate over $K$. Let us assume towards a contradiction that $D$ pulls back to an irreducible effective degree 4 divisor over $K$ on $X/w$, $X/vw$ and $X/v$. Write the pullback to $X$ as $D_1+D_2$, where $D_1$ and $D_2$ are degree 4 divisors over $K$.  Then for each $u\in \{v,w,vw\}$, we must have $u^*D_1=D_2$ and $u^*D_2=D_1$. In particular, $(vw)^*D_1=w^*v^*D_1=w^*D_2=D_1$, so that $D_1=D_2$ and $u^*D_1=D_1$ for each $u\in \{v,w,vw\}$. Now the pullback of $D$ to $X/u$ is the double of a degree 2 divisor over $K$, unless $u$ fixes $D_1$ pointwise. We conclude that $v,w$ and $vw$ fix $D_1$ pointwise. But then $D$ is irreducible of degree 4, a contradiction.
\end{proof}

\begin{proposition}
\label{quadprop}
Suppose that $P\in E_7$ or $P\in E_{21}$ is a quadratic point that is the image of a quartic point of $X_0(105)$. Then $P$ is the image of a quadratic point on one of the curves
\[
X_0(105)/w_5, \;\;X_0(105)/w_7,\;\; X_0(105)/w_{35},\;\; X_0(105)/w_{21}, \;\; X_0(105)/w_{105}.
\]
\end{proposition}
\[
\begin{tikzcd}
                                                                                      &                                                                                        &                                                                                    & X_0(105) \arrow[lldd] \arrow[ldd] \arrow[dd] \arrow[rdd] \arrow[rrdd]                 &                                                                                         &                                              \\
                                                                                      &                                                                                        &                                                                                    &                                                                                       &                                                                                         &                                              \\
                                                                                      & \frac{X_0(105)}{w_5} \arrow[ld,"\pi_3"] \arrow[d,"\pi_{21}"] \arrow[rd,"\pi_7"]              & \frac{X_0(105)}{w_{21}} \arrow[rd]                            & \frac{X_0(105)}{w_{7}} \arrow[rd] \arrow[d]                      & \frac{X_0(105)}{w_{105}} \arrow[d]                                 & \frac{X_0(105)}{w_{35}} \\
\frac{X_0(105)}{\langle w_5,w_3\rangle} \arrow[rrrdd] & \frac{X_0(105)}{\langle w_5,w_{21}\rangle} \arrow[rrdd] & \frac{X_0(105)}{\langle w_5,w_7\rangle} \arrow[rdd] & \frac{X_0(105)}{\langle w_7,w_{21}\rangle} \arrow[dd] & \frac{X_0(105)}{\langle w_7,w_{105}\rangle} \arrow[ldd] &                                              \\
                                                                                      &                                                                                        &                                                                                    &                                                                                       &                                                                                         &                                              \\
                                                                                      &                                                                                        &                                                                                    & X_0(105)^*=\mathbb{P}^1                                                               &                                                                                         &                                             
\end{tikzcd}
\]
\begin{proof}
We apply Lemma \ref{diagramlemma} to $X_0(105)$, first with involutions $w_5$ and $w_7$, then with involutions $w_5$ and $w_{21}$. 
\end{proof}
For coprime prime powers $p_1,\ldots,p_n$, we write $X_0(p_1\cdots p_n)^*=X_0(p_1\cdots p_n)/\langle w_{p_1},\ldots,w_{p_n}\rangle$.
\begin{lemma}
\label{Qcurvelemma}
Suppose that $p_1,p_2,p_3$ are pairwise coprime prime powers, and consider a non-cuspidal quartic point $P\in X_0(p_1p_2p_3)$ mapping to a rational point on \newline$X_0(p_1p_2p_3)^*\colonequals X_0(p_1p_2p_3)/\langle w_{p_1},w_{p_2},w_{p_3}\rangle$. Then $P$ corresponds to a $\Q$-curve. 
\end{lemma}
\begin{proof}
Define $W\colonequals \langle w_{p_1},w_{p_2},w_{p_3}\rangle$. The image of $P$ in $X_0(p_1p_2p_3)^*$ pulls back to the degree 8 divisor $\sum_{w\in W}w(P)$ on $X_0(p_1p_2p_3)$. A subset of these 8 (not necessarily distinct) points is the set of Galois conjugates of $P$, so each conjugate equals $w(P)$ for some $w\in W$. This means that the elliptic curve with $j$-invariant $j(P)$ is isogenous to all of its conjugates.
\end{proof}
By Proposition \ref{modw5prop} and \ref{quadprop} and Lemma \ref{Qcurvelemma}, we have now reduced our search for quartic points on $X_0(105)$ not corresponding to $\Q$-curves, to a search on five different curves, for the quadratic points not mapping to a rational point on $X_0(105)^*$. 

For $X_0(105)/w_5$ we have already verified (see the proof of Proposition \ref{modw5prop}) that each of the three quadratic points maps to a rational point on $E_7$ or $E_{21}$. The curves $X_0(105)/w_{35}$, $X_0(105)/w_{21}$, $X_0(105)/w_{105}$ and $X_0(105)/w_7$, have genera 3, 5, 5 and 7 respectively and none of them is hyperelliptic. As also their Mordell--Weil groups are finite, their sets of quadratic points must be finite.

Following Section \ref{finitemwgpsec}, we would like to determine their Mordell--Weil groups. For the latter three of those curves, however, this is challenging due to their high genus.  Instead, we make it easier for ourselves by considering the quotients 
\[
X_0(105)/\langle w_7,w_{105}\rangle \text{ and }X_0(105)/\langle w_7,w_{21}\rangle.
\]
Each quadratic point on $X_0(105)/w_{21},X_0(105)/w_{105}$ or $X_0(105)/w_7$ maps to a quadratic or rational point on one of these two curves. Both $X_0(105)/\langle w_7,w_{105}\rangle$ and $X_0(105)/\langle w_7,w_{21}\rangle$ are hyperelliptic of genus 3, so they have finitely many quadratic points, except for one infinite set coming from the double cover of $\P^1$.  Both hyperelliptic curves have an automorphism group of order 2, meaning the hyperelliptic involution equals the remaining Atkin--Lehner involution, and the map to $\P^1$ is the map to $X_0(105)^*$.

By Lemma \ref{Qcurvelemma}, the quadratic points coming from $\P^1$ correspond to $\Q$-curves, so it suffices to determine the quadratic points on the hyperelliptic curves $X_0(105)/\langle w_7,w_{105}\rangle$ and $X_0(105)/\langle w_7,w_{21}\rangle$ that do not come from $\P^1(\Q)$. First, we compute their models:
\begin{align*}
&X_0(105)/w_{35}:\;  -x_1^4 + 2x_1^3x_2 - 2x_1^2x_2^2 + x_1x_2^3 + 2x_1^3x_3 - 5x_1^2x_2x_3 \\
    &\;\;\;\;\;\;\;\;\;\;\;\;\;\;\;\;\;\;\;\;\;\;\;\;+ 4x_1x_2^2x_3 - x_2^3x_3 - 2x_1^2x_3^2 + 4x_1x_2x_3^2 - 3x_2^2x_3^2 + x_1x_3^3 - x_2x_3^3 =0,\\
  &  X_0(105)/\langle w_7,w_{105}\rangle: \;y^2 = -28x^7 - 36x^6 + 64x^5 + 21x^4 - 66x^3 + 39x^2 - 10x + 1 \text{ and } \\
   & X_0(105)/\langle w_7,w_{21}\rangle: \; y^2 = 20x^7 + 44x^6 - 16x^5 - 47x^4 + 22x^3 
    + 7x^2 - 6x + 1.
\end{align*}

\begin{proposition}\label{mwgpsprop}
The Mordell--Weil groups of the above curves are the following:
\begin{align*}
J(X_0(105)/w_{35})(\Q)&=\Z/4\Z \cdot [(1:0:1)-(0:1:0)]\oplus \Z/8\Z \cdot [(0:0:1)-(0:1:0)],\\
J(X_0(105)/\langle w_{7},w_{105}\rangle)(\Q)&=\Z/2\Z \cdot [D_1-2\cdot \infty] \oplus \Z/16\Z\cdot [(0,-1) - \infty] \text{ and }\\
J(X_0(105)/\langle w_7,w_{21}\rangle)(\Q)&=\Z/2\Z \cdot[D_2-2\cdot \infty] \oplus \Z/32\Z \cdot [(0,-1)-\infty],
\end{align*}
where $D_1$ and $D_2$ are the effective degree 2 divisors given by the equations $x^2-5/7x+1/7=0$ and $x^2+x-1=0$ respectively.
\end{proposition}
\begin{proof}
Note that all Mordell--Weil groups have rank 0 because $J_0(105)(\Q)$ has rank 0. First we compute the subgroups generated by the differences of rational points. Using the fact \cite{katz} that $J(X)(\Q)$ embeds in $J(X)(\F_p)$ for each of the curves $X$ and each prime $p>2$ of good reduction for $X$, we then reduce the possibilities. The primes 11, 17 and 79 (for $X_0(105)/\langle w_{7},w_{105}\rangle$) and 13, 31 and 43 (for $X_0(105)/\langle w_7,w_{21}\rangle$) tell us that the only possible Mordell--Weil groups are $\Z/2\Z\times \Z/16\Z$ and $\Z/16\Z$ for $X_0(105)/\langle w_{7},w_{105}\rangle$ and $\Z/2\Z \times \Z/32\Z$ and $\Z/32\Z$ for $X_0(105)/\langle w_7,w_{21}\rangle$. In both cases it thus suffices to compute the 2-torsion subgroup, which is easy for hyperelliptic curves. For odd-degree hyperelliptic curves, this is available in \texttt{Magma}.

Next, for $X_0(105)/w_{35}$, we find that the subgroup of its Mordell--Weil group generated by the four rational points (the images of the cusps of $X_0(105)$) is isomorphic to $H\colonequals \Z/4\Z\times \Z/8\Z$. Computing Mordell--Weil groups over $\F_p$ for $p\in \{ 11, 29, 107\}$, we limit the possiblities for $J(X_0(105)/w_{35})(\Q)$ to 
\[
H,\;\;\Z/2\Z\times H \text{ and } \Z/2\Z\times \Z/2\Z \times H.
\]
So here, too, it suffices to compute the 2-torsion. This curve is not hyperelliptic, but it is a plane quartic. For plane quartics, Bruin, Poonen and Stoll \cite[Section 12]{bps} found an algorithm for computing the Galois representation
$
\mathrm{Gal}(\overline{\Q}/\Q)\to \mathrm{Aut}_{\overline{\Q}}J[2]
$
by first computing the bitangents. This was used by \"Ozman and Siksek \cite[Section 5.5]{ozman} to compute $J_0(45)[2](\Q)$. The implementation used by  \"Ozman and Siksek for computing these bitangents (using the \texttt{EliminationIdeal} function in \texttt{Magma}) was too time-expensive for our curve. Instead, we will use the Euclidean algorithm. We work on the affine chart $x_2=1$. We substitute a generic line $x_3=-\beta x_1-\gamma$ into the quartic equation, giving us a degree 4 polynomial $p_{\beta,\gamma}(x_1)\in \Q[\beta,\gamma][x_1]$. The values $(\beta,\gamma)$ for which $p_{\beta,\gamma}(x_1)$ is a square, correspond to bitangents. When $p_{\beta,\gamma}(x_1)$ is a square, we have $\mathrm{deg}(\gcd(p_{\beta,\gamma},p_{\beta,\gamma}'))\geq 2$, so we can use the Euclidean algorithm to find a degree 1 polynomial $f\in \Q[\beta,\gamma][x_1]$ that must vanish identically in order to have $\mathrm{deg}(\gcd(p_{\beta,\gamma},p_{\beta,\gamma}'))\geq 2$. The equations in $\beta$ and $\gamma$ obtained by setting the coefficients of $f$ equal to zero, yield a 0-dimensional scheme of which the finitely many $\overline{\Q}$-points are swiftly found. We then check which of those pairs $(\beta,\gamma)$ indeed correspond to bitangents. This yields 27 of the 28 bitangents; the final bitangent line $x_1=x_2$ does not have non-zero coefficient for $x_3$. The field of definition of the bitangents is $K=\Q(\sqrt{5},\sqrt{-3},\sqrt{-7})$, so in particular $K=\Q(J[2])$.

We then use the implementation of \"Ozman and Siksek \cite{ozman} to compute the Galois representation. We conclude that the 2-torsion is $J(X_0(105)/w_{35})(\Q)[2]=\Z/2\Z\times \Z/2\Z$. 
\end{proof}
We call a quadratic point on a hyperelliptic curve \emph{isolated} when it does not map to a rational point on $\P^1$. 
\begin{proposition}
\label{quadptprop}
\begin{itemize}
\item[(i)] The curve $X_0(105)/w_5$ has three pairs of quadratic points, each of which maps to a rational point on $X_0(105)^*$.
\item[(ii)]The curve $X_0(105)/w_{35}$ has two pairs of quadratic points, each of which maps to a rational point on $X_0(105)^*$. 
\item[(iii)] Each quadratic point on $X_0(105)/w_7$ maps to a rational point on $X_0(105)^*$.
\item[(iv)] Each quadratic point on $X_0(105)/w_{105}$ maps to a rational point on $X_0(105)^*$.
\item[(v)] The curve $X_0(105)/w_{21}$ has exactly two pairs of quadratic points that do not map to a rational point on $X_0(105)^*$. They are defined over $\Q(\sqrt{5})$ and their inverse images in $X_0(105)$ are quartic points with quadratic $j$-invariants $632000 \pm 282880\sqrt{5}$.
\end{itemize}
\end{proposition}
\begin{proof}
For part (i) we refer to the proof of Proposition \ref{modw5prop}.  Given the Mordell--Weil groups, we can now compute the isolated quadratic points on $X_0(105)/w_{35}$, $X_0(105)/\langle w_7,w_{21}\rangle$ and $X_0(105)/\langle w_7,w_{105}\rangle$ using Lemma \ref{lem2.3}, again by computing the 1-dimensional Riemann--Roch spaces. For the two hyperelliptic curves, we exclude the zero class in the Mordell--Weil group, which corresponds to the infinitely many quadratic points coming from $\P^1=X_0(105)^*$ by Corollary \ref{rrcor}. We then pull back those quadratic points on the hyperelliptic curves to $X_0(105)/w_7$, $X_0(105)/w_{105}$ and $X_0(105)/w_{21}$, and check which have quadratic inverse images. 

Only on $X_0(105)/w_{21}$, this results in two pairs of quadratic points. Since $X_0(105)$ has genus 13, it is not computationally feasible to pull back these degree 2 divisors to $X_0(105)$, or to compute the $j$-invariant morphism on $X_0(105)$. Instead, we compute the Atkin--Lehner involution $w_5$ on $X_0(105)/w_{21}$, and check that both quadratic points $P$ satisfy $w_5(P)=P$. Let $Q$ be a quartic point on $X_0(105)$ mapping to one of these quadratic points. Then either $w_5(Q)=Q$ or $w_{105}(Q)=w_5w_{21}(Q)=Q$. 

We first suppose that $w_{105}(Q)=Q$. We compute the matrix $W$ defining the action of $w_{105}$ on the cusp forms defining our canonical model in $\P_{x_1,\ldots,x_{13}}^{12}$ for $X_0(105)$. We find that $\mathrm{Ker}(W-\mathrm{I})$ is 5-dimensional. From its basis, we find the equation $x_1-2x_{3}-2x_{4}-x_{5}=0$ satisfied by the $w_{105}$-fixed points of $X_0(105)$, along with 4 other equations. This particular equation is defined in terms of the first 5 coordinates, which are the coordinates corresponding to the cusp forms fixed by $w_5$. So the image of $Q$ on $X_0(105)/w_5$ must satisfy $x_1-2x_3-2x_4-x_5=0$. This equation defines a divisor on $X_0(105)/w_5$ that decomposes as a sum of two irreducible degree 4 divisors. Using Lemma \ref{squarechecklemma}, we find that these come from degree 8 points on $X_0(105)$.

We conclude that we must have $w_5(Q)=Q$. Then also the image $R$ of $Q$ on $X_0(5)$ must satisfy $w_5(R)=R$. We compute $X_0(5)$ as well as the action of $w_5$ on $X_0(5)$. The points on $X_0(5)$ fixed by $w_5$ are quadratic and have the claimed $j$-invariant.
\end{proof}
Finally, Propositions \ref{modw5prop}, \ref{quadprop} and \ref{quadptprop} together with Lemma \ref{Qcurvelemma} prove Theorem \ref{X0105thm}.

\section{$X(\mathrm{s}3,\mathrm{b}5,\mathrm{b}7)$}\label{Xs3b5b7sec}
The second modular curve we need to look at is $X(\mathrm{s}3,\mathrm{b}5,\mathrm{b}7)$. 
\begin{theorem}
\label{Xs3b5b7thm}
Let $E$ be an elliptic curve with quartic $j$-invariant that occurs as a quartic point on $X(\mathrm{s}3,\mathrm{b}5,\mathrm{b}7)$. Then $E$ is a $\Q$-curve, and in particular $E$ is modular.
\end{theorem}
In order to do computations with this curve, we make use of an isomorphism observed in \cite[p.\ 555]{conrad}:
\[
 X(\mathrm{b}p^2,*)/w_{p^2}\longrightarrow X(\mathrm{s}p,*),
\]
where $p$ is prime and $*$ denotes any prime-to-$p$ level structure. Pulled back to the upper half plane, this isomorphism is just $\tau\mapsto p\tau$. This yields $X(\mathrm{s}3,\mathrm{b}5,\mathrm{b}7)\simeq X_0(315)/w_9$, allowing us to study this curve using modular forms on $\Gamma_0(315)$. We find that $X(\mathrm{s}3,\mathrm{b}5,\mathrm{b}7)$ has genus 21. Moreover, evaluating $L(f,1)$ for the eigenforms in the $+1$-eigenspace for $w_9$ of $S_2(\Gamma_0(315))$  suggests (but does not prove) that the Mordell--Weil groups of several components of $J(\mathrm{s}3,\mathrm{b}5,\mathrm{b}7)$ have positive rank, c.f.\ Section \ref{ranksec}. The genus of $X(\mathrm{s}3,\mathrm{b}5,\mathrm{b}7)$ is too large to do explicit computations, so one hopes to find a quotient (such as $X_0(105)/w_5$ in the previous section) that is more amenable to computations. We found each quotient to be either too complicated to study computationally (high genus with positive rank Mordell--Weil group) or too simple to distinguish between quartic points (genus less than 4). 

We resolve this issue by combining information from two small quotients: $X_0(35)$ and $X(\mathrm{s}3,\mathrm{b}7)\simeq X_0(63)/w_9$. Both quotients are hyperelliptic of genus 3 and, conveniently, the Mordell--Weil groups of their Jacobians are finite. The small genus means that both curves have infinitely many quartic points.
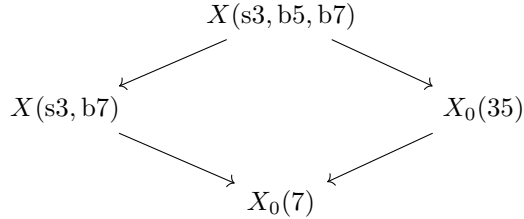
\[
\begin{tikzcd}
                       & X(\mathrm{s}3,\mathrm{b}5,\mathrm{b}7) \arrow[ld] \arrow[rd] &                    \\
X(\mathrm{s}3,\mathrm{b}7) \arrow[rd] &                                    & X_0(35) \arrow[ld] \\
                       & X_0(7)                             &                   
\end{tikzcd}
\]

Note that $X_0(7)$ has genus 0. We compute models for the curves and maps in the lower half of the diagram via modular forms. The models for $X_0(35)$ and $X_0(7)$, as well as the map between them, can be computed in \texttt{Magma} using the Small Modular Curves package. The map from $X_0(63)/w_9$ to $X_0(7)$ is not implemented, but we computed it using the method described in Section \ref{modelssec}. We display here only the curves. 
\begin{align*}
    X_0(35):&\; y^2 = x^8 - 4x^7 - 6x^6 - 4x^5 - 9x^4 + 4x^3 - 6x^2 + 4x + 1,\\
    X(\mathrm{s}3,\mathrm{b}7):&\; y^2 = x^8 + 6x^7 + 23x^6 + 36x^5 + 57x^4 + 36x^3 + 23x^2 + 6x + 1,\\
    X_0(7)=\P^1.
\end{align*}
On these hyperelliptic curves, we write $\infty^+=(1:1:0)$ and $\infty^-=(1:-1:0)$ (in weighted projective space). The cusps of $X_0(35)$ are $\infty^+$, $\infty^-$, $(0,1)$ and $(0,-1)$, and the cusps of $X(\mathrm{s}3,\mathrm{b}7)$ are (also) $\infty^+$, $\infty^-$, $(0,1)$ and  $(0,-1)$. 

As observed e.g.\ in \cite[Remark in Section 5.3]{freitas}, there is an exceptional inclusion (up to conjugation) $C_{\mathrm{ns}}^+(3)\subset C_{\mathrm{s}}^+(3)$, and this inclusion of subgroups has index 2 after intersecting with $\mathrm{SL}_2(\F_3)$. This means that there is a morphism of degree 2
\[
X(\mathrm{s}3,\mathrm{b}7)\to X(\mathrm{ns}3,\mathrm{b}7).
\]
Each degree 2 morphism of curves is the quotient by an involution. We call this involution $\phi_3:\; X(\mathrm{s}3,\mathrm{b}7)\to X(\mathrm{s}3,\mathrm{b}7)$.
\begin{lemma}\label{hyperellipticlemma}\begin{itemize} \item[(i)] The hyperelliptic involution on $X_0(35)$ is $w_{35}$.
\item[(ii)] The hyperelliptic involution on $X(\mathrm{s}3,\mathrm{b}7)$ is $\phi_3w_7$. 
\end{itemize}
\end{lemma}
\begin{proof}
Part (i) is due to Ogg \cite{ogg}. For part (ii), we can check in \texttt{Magma} that $X(\mathrm{s}3,\mathrm{b}7)$ has an automorphism group isomorphic to $\Z/2\Z\times \Z/2\Z$. We then note that Le Hung \cite[Remark 5.2]{lehung} showed that $X(\mathrm{ns}3,\mathrm{b}7)$ has genus 2. We can also check that the subspace of $S_2(\Gamma_0(63))$ fixed by $w_7$ and $w_9$ is 1-dimensional, or alternatively use the \texttt{ModularCurveQuotient} function in \texttt{Magma} to check that $X_0(63)/\langle w_9,w_7\rangle$ (which is isomorphic to $X(\mathrm{s}3,\mathrm{b}7)/w_7$) has genus 1. This shows that $w_7\neq \phi_3$, and thus $\mathrm{Aut}(X(\mathrm{s}3,\mathrm{b}7))=\langle \phi_3, w_7\rangle$. It also shows that neither $w_7$ nor $\phi_3$ is the hyperelliptic involution, leaving $\phi_3w_7$ as the only remaining candidate.
\end{proof}
It is possible to compute $w_7$ and $\phi_3$ on $X(\mathrm{s}3,\mathrm{b}7)$ explicitly using the methods in \cite{boxmodels} and verify by computation that $w_3\phi_7$ is the hyperelliptic involution, but we have not done this. 

Next, we need to compute Mordell--Weil groups. 

\begin{proposition}
\label{mwgrpsprop}
The Mordell--Weil groups of $J_0(35)$ and $J(\mathrm{s}3,\mathrm{b}7)$ are
\[
J_0(35)(\Q)=\Z/24\Z \cdot [\infty^--\infty^+]\oplus \Z/2\Z\cdot [3\cdot (0,-1)- 3\cdot \infty^+]
\]
and
\begin{align*}
J(\mathrm{s}3,\mathrm{b}7)(\Q) &= \Z/24\Z \cdot [ 2\cdot \infty^++(0,1)-3\cdot \infty^-]\\& \oplus \Z/2\Z\cdot [-9\infty^++(0,1)+8\cdot \infty^-]
\end{align*}
\end{proposition}
\begin{proof}
For $X_0(35)$, we use the prime 3 to embed $J_0(35)(\Q)$ in $J_0(35)(\F_3)\simeq \Z/24\Z\times \Z/2\Z$. We find that the images of the claimed generators indeed have the right order and generate $J_0(35)(\F_3)$. This agrees with the Mordell--Weil group $J_0(35)(\Q)$ found in \cite[Lemma 12.1]{freitas}.

For $X(\mathrm{s}3,\mathrm{b}7)$ we embed its Mordell--Weil group in $J(\mathrm{s}3,\mathrm{b}7)(\F_5)\simeq (\Z/2\Z)^3\times \Z/24\Z$ and $J(\mathrm{s}3,\mathrm{b}7)(\F_{11})\simeq \Z/4\Z \times \Z/264\Z$. The group generated by the images of the claimed generators is isomorphic to $\Z/24\Z\times \Z/2\Z$, with these generators having orders 24 and 2 respectively, as desired.
\end{proof}
For each of the two hyperelliptic curves, let $\infty=\infty^++\infty^-$ denote the (degree 2) divisor at infinity. Also, for each $Y\in \{X(\mathrm{s}3,\mathrm{b}7),X_0(35)\}$, we define the Abel--Jacobi map by
\[
\iota_Y:\;Y^{(4)}\to J(Y),\;\; D\mapsto [D-2\cdot \infty].
\]
\begin{proposition}
For each $Y\in \{X_0(35),X_0(63)/w_9\}$, the Riemann--Roch spaces
\[
L(D+2\cdot \infty) \text{ for } D\in J(Y)(\Q)
\]
are 2-dimensional, except for when $D=0$, in which case $L(D+2\cdot \infty)=L(2\cdot\infty)$ is 3-dimensional and generated by $1,x,x^2$.
\end{proposition}
\begin{proof}
It is clear that $1,x,x^2\in L(2\cdot \infty)$ and $\ell(2\cdot\infty)\leq 3$ by Clifford's theorem, from which we deduce the final statement. We also verified this computationally in \texttt{Magma}. In the other cases we simply computed the dimension of the Riemann--Roch spaces computationally.
\end{proof}
For $Y\in \{X(\mathrm{s}3,\mathrm{b}5),X_0(35)\}$, we define $\pi_Y: X(\mathrm{s}3,\mathrm{b}5,\mathrm{b}7)\to Y$ to be the projection map. 
\begin{proposition}
\label{zerozeroprop}
Consider $D\in X(\mathrm{s}3,\mathrm{b}5,\mathrm{b}7)^{(4)}(\Q)$ supported on an elliptic curve $E$ defined over a quartic field and suppose that the push-forwards $\pi_{Y,*}D$ for $Y\in \{X(\mathrm{s}3,\mathrm{b}7),X_0(35)\}$ both satisfy $[\pi_{Y,*}D-2\cdot \infty]=0\in J(Y)(\Q)$. Then $E$ is a $\Q$-curve.
\end{proposition}
\begin{proof}
For each $Y\in \{X(\mathrm{s}3,\mathrm{b}7),X_0(35)\}$, we find that there is a function $f$ on $Y$, which is a quadratic function of $x$ and hence factors via $x: Y\to \P^1$, such that $\pi_{Y,*}D=\div(f)+2\cdot \infty$. Writing $f=g\circ x$, we find $\pi_{Y,*}D=x^*(\div(g)+2\cdot \infty_{\P^1})$. Now assume that $D$ is generated by a quartic point $P$, defined over the degree 4 number field $K$. We assume that $P$ is supported on an elliptic curve $E/K$.  By the above and by Lemma \ref{hyperellipticlemma}, the images of $P$ in both $X(\mathrm{s}3,\mathrm{b}7)/\phi_3w_7$ and $X_0(35)/w_{35}$ are quadratic, because $\div(g)+2\cdot \infty_{\P^1}$ is a divisor of degree 2. Note that $\phi_3$ acts only on the 3-level structure and leaves underlying elliptic curves fixed, whereas each Atkin--Lehner involution changes an underlying elliptic curve defined over a field $F$ by an $F$-isogeny. In particular, the hyperelliptic involution on each $Y$ maps $\pi_Y(P)$ to a conjugate of $\pi_Y(P)$, supported on a conjugate $E^Y$ of $E$ which is $K$-isogenous to $E$. Suppose first that $j(E^{X_0(35)})\neq j(E^{X(\mathrm{s}3,\mathrm{b}7)})$. Then $E$ is isogenous by a $K$-isogeny to two distinct Galois conjugates. In particular, these two conjugates of $E$ are also defined over $K$ (hence so must the final conjugate), and $K$ is Galois. Suppose that we have $\sigma,\tau\in \mathrm{Gal}(K/\Q)$ such that $\phi_3w_7$ maps $\pi_{X(\mathrm{s}3,\mathrm{b}7)}(P)$ to $\pi_{X(\mathrm{s}3,\mathrm{b}7)}(P)^{\sigma}$ and $w_{35}$ maps $\pi_{X_0(35)}(P)$ to $\pi_{X_0(35)}(P)^{\tau}$. Since $w_{35}$ is defined over $\Q$, it maps $\pi_{X_0(35)}(P^{\sigma})$ to $w_{35}(\pi_{X_0(35)}(P))^{\sigma}=P^{\tau\sigma}$, supported on $E^{\tau\sigma}$. This, finally, defines an isogeny between $E^{\sigma}$ and $E^{\tau\sigma}$, proving that $E$ is isogenous to all of its conjugates and is therefore a $\Q$-curve. 

%

If $j(E^{X_0(35)})= j(E^{X(\mathrm{s}3,\mathrm{b}7)})$ then $\phi_3w_7(\pi_{X(\mathrm{s}3,\mathrm{b}7)}(P))$ and $w_{35}(\pi_{X_0(35)}(P))$  are supported on the same elliptic curve, so $Q=w_7(P)$ satisfies $j(Q)=j(w_5(Q))$. Thus $Q$ gives rise to a quartic point $R$ on $X_0(5)$ such that $R$ and $w_5(R)$ have the same $j$-invariant. We can explicitly compute the $j$-invariant $j$ and Atkin--Lehner involution $w_5$ on $X_0(5)\simeq \P^1$ and find the zeros of $j-j\circ w_5$. All of them are defined over quadratic fields. 
\end{proof}
Again consider $Y\in \{X(\mathrm{s}3,\mathrm{b}7),X_0(35)\}$. Note that an $r$-dimensional Riemann--Roch space yields an $(r-1)$-dimensional linear system because the divisors depend only on the functions up to scaling. 

Now consider a pair $(D_1,D_2)\in J_0(35)(\Q)\times J(\mathrm{s}3,\mathrm{b}7)(\Q)$. We compute explicit bases for the Riemann--Roch spaces of $D_1+2\cdot \infty$ and $D_2+2\cdot\infty$; let us call these $f_0,\ldots,f_n$ and $g_0,\ldots,g_m$ respectively.  The effective degree 4 divisors mapping to $D_1$ resp. $D_2$ under $\iota$ can then be parametrised as follows, see Lemma \ref{lem2.3}:
\begin{align*}
\iota_{X_0(35)}^{-1}(D_1)&=\{D_1+2\cdot \infty+\mathrm{div}(t_0f_0+\ldots t_n f_n)\mid (t_0:\ldots:t_n)\in \P^n(\Q)\} \text{ and }\\
\iota_{X(\mathrm{s}3,\mathrm{b}7)}^{-1}(D_2)&=\{D_2+2\cdot \infty+\mathrm{div}(s_0g_0+\ldots s_m g_m)\mid (s_0:\ldots:s_m)\in \P^m(\Q)\}.
\end{align*}
 By the Riemann--Roch theorem, when $D_1\neq 0$ and $D_2\neq 0$, we have $n=m=1$. This agrees with our intuition because $1=\dim Y^{(4)}-\dim J(Y)$.  When $D_1=0$ and $D_2\neq 0$ (or the other way around), we have $n=2$ and $m=1$ (resp. $m=2$ and $n=1$). Thanks to Proposition \ref{zerozeroprop}, we do not need to worry about the final possibility $D_1=D_2=0$. 

We now base-change the curves to $\Q(t_0,\ldots,t_n)$ and $\Q(s_0,\ldots,s_m)$ respectively. Write $\mathbf{t}=(t_0,\ldots,t_n)$ and $\mathbf{s}=(s_0,\ldots,s_m)$. We consider the divisors
\begin{align*}
D_1(\mathbf{t})\colonequals D_1+2\cdot\infty +\mathrm{div}(t_0f_0+\ldots+t_nf_n) \text{ and}\\
D_2(\mathbf{s})\colonequals D_2+2\cdot \infty +\mathrm{div}(s_0g_0+\ldots+s_mg_m).
\end{align*}
Next, we define 
\[
S_{D_1,D_2}\subset \P^n\times \P^m
\]
to be the subscheme determined as those $\mathbf{t}\in \P^n$ and $\mathbf{s}\in \P^m$ such that 
\[
(\pi_{X_0(35)})_*D_1(\mathbf{t})=(\pi_{X(\mathrm{s}3,\mathrm{b}7)})_*D_2(\mathbf{s}).
\]
\begin{proposition}
\label{pushforwardprop}
There are exactly two pairs $(P,Q)$ of  $\Q$-rational degree 4 divisors $P$ on $X_0(35)$ and $Q$ on $X(\mathrm{s}3,\mathrm{b}7)$ with irreducible pushforward to $X_0(7)$, such that $P=D_1(\mathbf{t})$ and $Q=D_2(\mathbf{s})$ for some   $(D_1,D_2)\in J(X_0(35))(\Q)\times J(X(\mathrm{s}3,\mathrm{b}7))(\Q)\setminus \{(0,0)\}$ and $(\mathbf{t},\mathbf{s})\in S_{D_1,D_2}(\Q)$. Each of those pairs has quadratic $j$-invariant $632000\pm282880\sqrt{5}$.
\end{proposition}
\begin{remark}
If this $j$-invariant looks familiar, that is because we found the same $j$-invariant for the quartic points on $X_0(105)$ determined in Proposition \ref{quadptprop} (v). The field of definition of the corresponding points on $X_0(35)$, $X(\mathrm{s}3,\mathrm{b}7)$ and $X_0(105)$ is $\Q(\sqrt{5},i)$. The corresponding elliptic curve has complex multiplication by an order of discriminant -20.
\end{remark}
\begin{proof}
We enumerate the possibilities for $(D_1,D_2)$ using Proposition \ref{mwgrpsprop}. For each pair $(D_1,D_2)$, we need to compute $S_{D_1,D_2}(\Q)$. To that end, we compute the push-forward of $D_1(\mathbf{t})$ over the field $\Q(\mathbf{t})$. Because $X_0(7)=\P^1_{u,v}$, we view $X_0(35)\to X_0(7)$ as an element of the function field of $X_0(35)$. At each irreducible component of $D_1(\mathbf{t})$ over $\Q(\mathbf{t})$, we evaluate this function, and we determine the minimal polynomial (over $\Q(\mathbf{t})$) of the result. This gives us (after homogenising) the monic degree 4 equation 
\[
p_{\mathbf{t}}(u,v)\colonequals u^4+a_3(\mathbf{t})u^3v+a_2(\mathbf{t})u^2v^2+a_1(\mathbf{t})uv^3+a_0(\mathbf{t})v^4=0,
\]
defining $(\pi_{X_0(35)})_*D_1(\mathbf{t})$,
with coefficients in $\Q(\mathbf{t})$. We do the same for $D_2(\mathbf{s})$, which yields a monic degree 4 equation $q_{\mathbf{s}}(u,v)\colonequals u^4+b_3(\mathbf{s})u^3v+\ldots+b_0(\mathbf{s})v^4=0$. Then  we equate each of the four coefficients of $p_{\mathbf{t}}$ and $q_{\mathbf{s}}$. We note that each coefficient is a \emph{rational} function in $\mathbf{t}$, resp. $\mathbf{s}$. So when equating coefficients we need to multiply by denominators. This yields 4 equations in $t_0,\ldots,t_n$ and $s_0,\ldots,s_m$ defining a scheme $T_{D_1,D_2}$ such that $S_{D_1,D_2}\subset T_{D_1,D_2}\subset \P^n\times \P^m$. Indeed, by construction $S_{D_1,D_2}\subset T_{D_1,D_2}$. However, if $\mathbf{t}$ is a zero of the denominator of $a_i(\mathbf{t})$ for some $i$, and similarly $\mathbf{s}$ is a zero of the denominator of $b_i(\mathbf{s})$, then the equation $a_i(\mathbf{t})=b_i(\mathbf{s})$ is automatically satisfied, regardless of whether $D_1(\mathbf{t})$ and $D_2(\mathbf{s})$ have the same pushforward to $X_0(7)$. This can lead to extra irreducible components on $T_{D_1,D_2}$ that do not occur on $S_{D_1,D_2}$.  

Next, we decompose $T_{D_1,D_2}$ into irreducible components. We find that each component is at most 1-dimensional. In every case where a component is 1-dimensional, we find that indeed the denominator $d$ of one of the $a_i(\mathbf{t})$ or $b_i(\mathbf{s})$ vanishes identically on that component. This means that  $(\pi_{X_0(35)})_*D_1(\mathbf{t})$ or $(\pi_{X(\mathrm{s}3,\mathrm{b}7)})_*D_2(\mathbf{s})$ is reducible, because $d\cdot p_{\mathbf{t}}$, resp. $d\cdot q_{\mathbf{s}}$, is divisible by $v$. Reducible on $X_0(7)$ means that the $j$-invariant takes rational, quadratic or cubic values, and we can thus ignore such components. 

On each 0-dimensional component of $T_{D_1,D_2}$, we determine all rational points explicitly. Again we remove points $(\mathbf{s},\mathbf{t})$ where a denominator of a coefficient vanishes, and for the same reason those where $a_0(\mathbf{t})$ or $b_0(\mathbf{s})$ vanishes. From the values of $\mathbf{s}$ and $\mathbf{t}$ we find two explicit pairs of divisors. With the \texttt{SmallModularCurves} package in \texttt{Magma}, we compute the $j$-invariant function on $X_0(35)$ and evaluate it at the two points. 
\end{proof}
Finally, Propositions \ref{zerozeroprop} and \ref{pushforwardprop} together prove  Theorem \ref{Xs3b5b7thm}.
\section{$X(\mathrm{b}5,\mathrm{ns}7)$}\label{Xb5ns7sec}
\subsection{Overview}\label{sec51}Next, we turn our attention to the curve $X(\mathrm{b}5,\mathrm{ns}7)$. We first note that there are morphisms $X(\mathrm{b}3,\mathrm{b}5,\mathrm{e}7)\to X(\mathrm{b}5,\mathrm{e}7)$ and $X(\mathrm{s}3,\mathrm{b}5,\mathrm{e}7)\to X(\mathrm{b}5,\mathrm{e}7)$ forgetting the level 3 structure. On $X(\mathrm{b}5,\mathrm{e}7)$, we have an Atkin--Lehner involution $w_5$, and we have the degree 2 quotient map $X(\mathrm{b}5,\mathrm{e}7)\to X(\mathrm{b}5,\mathrm{ns}7)$ (recall that the $\mathrm{e}7$-structure is defined by an index 2 subgroup of the normaliser of a non-split Cartan subgroup of $\mathrm{GL}_2(\F_7)$). Each degree 2 map of curves is the quotient by an involution, which we call 
\[
\phi_7:\; X(\mathrm{b}5,\mathrm{e}7)\to X(\mathrm{b}5,\mathrm{e}7).
\]
This is the morphism determined by a matrix in $\Gamma_0(5)\cap \Gamma(\mathrm{ns}7)\setminus \Gamma_0(5)\cap \Gamma(\mathrm{e}7)$, c.f.\ \cite[Section 5]{boxmodels}. As $w_5$ descends to a morphism on $X(\mathrm{b}5,\mathrm{ns}7)=X(\mathrm{b}5,\mathrm{e}7)/\phi_7$, the two involutions $w_5$ and $\phi_7$ commute. We thus obtain the following commutative diagram of degree 2 maps between modular curves
\[
\begin{tikzcd}
                                         & {X(\mathrm{b}5,\mathrm{e}7)} \arrow[d] \arrow[ld] \arrow[rd] &                                             \\
{X(\mathrm{b}5,\mathrm{ns}7)} \arrow[rd, "\rho"] & {X(\mathrm{b}5,\mathrm{e}7)/\phi_7w_5} \arrow[d]             & {X(\mathrm{b}5,\mathrm{e}7)/w_5} \arrow[ld] \\
                                         & {X(\mathrm{b}5,\mathrm{ns}7)/w_5}                            &                                            
\end{tikzcd}.
\]
Now each quartic point on $X(\mathrm{b}3,\mathrm{b}5,\mathrm{e}7)$ or $X(\mathrm{s}3,\mathrm{b}5,\mathrm{e}7)$ maps to a quartic point on $X(\mathrm{b}5,\mathrm{ns}7)$, so it would suffice to determine all quartic points on this curve instead.

The curve $X(\mathrm{b}5,\mathrm{ns}7)$ was studied by Le Hung \cite{lehung}, and later by Derickx, Najman and Siksek \cite{derickx}, who found a planar model in $\P^2_{u,v,w}$: 
\begin{align*}
X(\mathrm{b}5,\mathrm{ns}7):\;5u^6 &- 50u^5v + 206u^4v^2 - 408u^3v^3 + 
    321u^2v^4 + 10uv^5 - 100v^6 + 9u^4w^2 - 
    60u^3vw^2\\& + 80u^2v^2w^2 + 48uv^3w^2 + 
    15v^4w^2 + 3u^2w^4 - 10uvw^4 + 6v^2w^4
    - 1w^6=0.
\end{align*}
Here the Atkin--Lehner involution $w_5$ maps $(u:v:w)\mapsto (u:v:-w)$.  The curve has genus 6 and its Jacobian's Mordell--Weil group has rank 2. Its quotient $C\colonequals X(\mathrm{b}5,\mathrm{ns}7)/w_5$ is hyperelliptic of genus 2 with Jacobian also of Mordell--Weil rank 2. It has a model given by
\begin{align}\label{Ceqn}
C:\;y^2 = x^6 + 2x^5 + 7x^4 - 4x^3 + 3x^2 - 10x 
    + 5.
\end{align}
Given the genus and the rank of $X(\mathrm{b}5,\mathrm{ns}7)$, one might suspect $X(\mathrm{b}5,\mathrm{ns}7)^{(4)}(\Q)$ to be finite; the degree 2 map $\rho\colon X(\mathrm{b}5,\mathrm{ns}7)\to  C$ disproves this, however: $C^{(2)}(\Q)$ is infinite. The best we can hope for, is the following.

\begin{theorem}\label{Xb5ns7thm}
For each $i\in \{0,\ldots,4\}$, let $\mathcal{P}_i$ be the effective divisor that is the sum of $P_i$ and its Galois conjugates, where $P_0,\ldots,P_4$ are defined in Table \ref{tableb5ns7}. We have
\[
X(\mathrm{b}5,\mathrm{ns}7)^{(4)}(\Q)=\{\mathcal{P}_1,w_5^*\mathcal{P}_1,\ldots,\mathcal{P}_4,w_5^*\mathcal{P}_4\}\cup (\mathcal{P}_0+\rho^*C(\Q))\cup (w_5^*\mathcal{P}_0+\rho^*C(\Q))\cup \rho^*C^{(2)}(\Q).
\]
None of the fields of definition of $P_1,\ldots,P_4$ is totally real. 
\end{theorem}
\begin{remark}
Note that the elliptic curves corresponding to $P_2,w_5(P_2),P_3,w_5(P_3)$ are modular because they have complex multiplication. On $X_0(105)$ and $X(\mathrm{s}3,\mathrm{b}5,\mathrm{b}7)$, we showed that all quartic points correspond to modular elliptic curves, but here we do need to use the totally real condition to exclude $P_1$, $w_5(P_1)$, $P_4$ and $w_5(P_4)$.
\end{remark}

\begin{table}[h!]\label{tableb5ns7}
    \centering
    \begin{tabularx}{\textwidth}{c c >{\centering\arraybackslash}m{4.2cm}  >{\centering\arraybackslash}m{4cm} c}
        Name & $\mathrm{minpol}(\theta)$ & Coordinates & $j$-invariant & CM   \\ [0.5ex] 
        \hline \hline  
        $P_0$  &  $x^2-5$ & $(\theta - 5: \theta - 3: 2)$ &  1728& -4   \\ \hline 
    $w_5(P_0)$  & $x^2-5$& $(-\theta + 5: -\theta + 3: 2)$ & $-9845745509376\theta + 22015749613248$ & -100 \\ \hline
        $P_1$ & $x^4 + x^3 + 
    2x^2 - 1$ &  $(\theta + 1 : \theta : 1)$ & \tiny{$-879563721911\theta^3 - 225656459826\theta^2 - 
    1591364274227\theta + 1183091832488$} & NO  \\ \hline 
    $w_5(P_1)$ & $x^4 + x^3 + 
    2x^2 - 1$ &  $(\theta + 1 : \theta : -1)$ & \tiny{$-32369185233\theta^3 - 111707933554\theta^2 - 17543368031\theta
    + 35398801591$} & NO  \\\hline
    $P_2$ & $x^4-5$ &  $(11\theta^3 + 24\theta^2 - 35\theta - 82
: (7\theta^3 + 4\theta^2 - 11\theta - 24) : -62)$ & 287496 & -16  \\\hline
    $w_5(P_2)$ & $x^4 -5$ &  $(11\theta^3 + 24\theta^2 - 35\theta - 82
: (7\theta^3 + 4\theta^2 - 11\theta - 24) : 62)$ & \tiny{$-144957130139122438009154688\theta^3 + 
    216761467882862592971592192\theta^2 - 324133996814361318579017184\theta+ 
    484693377088718179461346056$} & -400\\\hline
    $P_3$ & $x^4 - x^3 + x^2
    + 4x - 4$ &  $(3\theta^3 - 3\theta^2 + 5\theta + 10 : 
\theta^3 - \theta^2 + 3\theta + 2) : -4)$ & -32768 & -11  \\\hline
    $w_5(P_3)$ & $x^4 - x^3 + x^2
    + 4x - 4$ &  $(3\theta^3 - 3\theta^2 + 5\theta + 10 : 
\theta^3 - \theta^2 + 3\theta + 2) : 4)$ & \tiny{$-7084562494756740890624\theta^3 + 
    1098708046310747471872\theta^2 - 6156247379826216271872\theta - 
    33539756888104194277376$} & -275 \\\hline
    $P_4$ & $x^4 - x^3 + 
    2x^2 - x - 2$ &  $(2\theta^2 - 2\theta : \theta^3 - \theta^2 +
    2\theta - 3 : -2)$ & \tiny{$(3826190688474785295\theta^3 + 
    749846172198673375\theta^2 + 8549180309143366704\theta + 6398434116827442268)/128$} & NO  \\\hline
    $w_5(P_4)$ & $x^4 - x^3 + 
    2x^2 - x - 2$ &  $(2\theta^2 - 2\theta : \theta^3 - \theta^2 +
    2\theta - 3 : 2)$ & \tiny{$(-248974618492363393\theta^3 + 
    412989831164677599\theta^2 - 769944755554282000\theta + 756153558218027868)/34359738368$} & NO\\
    \end{tabularx}
    \caption{All isolated quadratic and all isolated quartic points on $X(\mathrm{b}5,\mathrm{ns}7)$, up to Galois conjugacy.}
\end{table}

\begin{corollary}\label{X1X2cor}
Each quartic point on $X(\mathrm{b}5,\mathrm{e}7)$ whose image in $X(\mathrm{b}5,\mathrm{ns}7)$ is not a Galois conjugate of $P_1$, $w_5(P_1)$,\ldots,$P_4$ or $w_5(P_4)$ maps to a quadratic point on one of $X(\mathrm{b}5,\mathrm{ns}7)$, $X(\mathrm{b}5,\mathrm{e}7)/w_5$, $X(\mathrm{b}5,\mathrm{e}7)/\phi_7w_5$. 
\end{corollary}
\begin{proof}
Apply Lemma \ref{diagramlemma} to $X(\mathrm{b}5,\mathrm{e}7)$ with involutions $w_5$ and $\phi_7$, and use Theorem \ref{Xb5ns7thm}. 
\end{proof}
We note that each quadratic point on $X(\mathrm{b}5,\mathrm{ns}7)$ has quadratic $j$-invariant, and thus corresponds to a modular elliptic curve. So Theorem \ref{Xb5ns7thm} reduces the study of the quartic points on $X(\mathrm{b}5,\mathrm{e}7)$ with totally real quartic $j$-invariants to the study of the quadratic points on the curves $X(\mathrm{b}5,\mathrm{e}7)/w_5$ and $X(\mathrm{b}5,\mathrm{e}7)/\phi_7w_5$. These curves (of genera 5 and 8 respectively) are studied in Section \ref{X1andX2sec}. In the next sections, we use the new Chabauty method for symmetric powers developed in \cite{bgg} (and introduced in Section \ref{infmwgpsec}) in combination with the Mordell--Weil sieve introduced in Section \ref{mwsieve} to prove Theorem \ref{Xb5ns7thm}. More precisely, in Sections \ref{deg4chabsec}, \ref{Xb5ns7mwgp} and \ref{ptssec}, we determine the input of the sieve, after which we prove Theorem \ref{Xb5ns7thm} in Section \ref{Xb5ns7thmproofsec}.

\subsection{Quadratic points}
First, for completeness we mention the quadratic points on $X(\mathrm{b}5,\mathrm{ns}7)$. Derickx, Najman and Siksek \cite{derickx} showed that $X(\mathrm{b}5,\mathrm{ns}7)^{(3)}(\Q)=\{c_0,c_{\infty}\}$, where $c_{0}$ and $c_{\infty}$ are the Galois orbits of the cusps. In coordinates, $c_0$ is the effective degree 3 divisor generated by
\[
(4\eta^2-21\eta-7:\eta^2-7\eta :-14),\;\; \eta=\zeta_7+\zeta_7^{-1}\in \Q(\zeta_7)^+
\]
and $c_{\infty}=w_5^*c_0$. In particular, $X(\mathrm{b}5,\mathrm{ns}7)$ has no rational points. We find two isolated pairs of quadratic points $P_0$ and $w_5(P_0)$; let $\mathcal{P}_0$ be the sum of $P_0$ and its Galois conjugate. Each  $\mathcal{P}\in X(\mathrm{b}5,\mathrm{ns}7)^{(2)}(\Q)$ satisfies  $\mathcal{P}+\mathcal{P}_0 \in X(\mathrm{b}5,\mathrm{ns}7)^{(4)}(\Q)$, so it follows from Theorem \ref{Xb5ns7thm} that $\mathcal{P}\in \{\mathcal{P}_0,w_5^*\mathcal{P}_0\}\cup \rho^*C(\Q)$. We conclude that
\[
X(\mathrm{b}5,\mathrm{ns}7)^{(2)}(\Q)=\{\mathcal{P}_0,w_5^*\mathcal{P}_0\}\cup \rho^*C(\Q).
\]

\begin{remark}\label{quadrmk}
In \cite{freitas}, Freitas, Le Hung and Siksek proved modularity of all elliptic curves over real quadratic fields by studying the quadratic points on seven modular curves. Now, thanks to the recent theorems of Thorne \cite{thorne} and Kalyanswamy \cite{kalyanswamy}, it suffices by Theorem \ref{thm1.1} to show that all real quadratic points on $X(\mathrm{b}5,\mathrm{ns}7)$ and $X_0(35)$ correspond to modular elliptic curves, and separately prove that elliptic curves over $\Q(\sqrt{5})$ are modular. The latter two tasks are covered in \cite{freitas}, while we have just shown all quadratic points $P$ on $X(\mathrm{b}5,\mathrm{ns}7)$ to be modular: either $P$ has CM, or $P$ maps to a rational point on $C$ and is therefore a $\Q$-curve (it satisfies $P^{\sigma}=w_5(P)$). This provides a shorter proof for modularity of elliptic curves over real quadratic fields.
\end{remark}

Finally, we list the found rational points on $C$, together with the $j$-invariants of their inverse images on $X(\mathrm{b}5,\mathrm{ns}7)$. These inverse images are in each case defined over a quadratic field (i.e.\ not rational), so the Galois orbit of this $j$-invariant is well-defined. 
\begin{table}[h!]
    \centering
    \begin{tabularx}{\textwidth}{c c >{\centering\arraybackslash}X c  c}
        Name & Coordinates on $C$ & $j$-invariant of inverse image in $X(\mathrm{b}5,\mathrm{ns}7)$ & CM   \\ [0.5ex] 
        \hline \hline  
        $Q_1$  &  $\infty^-$ &  287496& -16   \\ \hline 
    $Q_2$  &  $\infty^+$ & $(85995\sqrt{5} - 191025)/2$ & -15  \\\hline
    $Q_3$  &  $(1 , -2 )$ & -32768 & -11  \\\hline
    $Q_4$  &  $(1 , 2 )$ & $184068066743177379840\sqrt{5} - 
    411588709724712960000$ & -235\\\hline
    $Q_5$ &  $(1/2 , -7/2 )$ & 1728 & -4  \\\hline
    $Q_6$ &   $(1/2 , 7/2)$ & $(-16554983445\sqrt{5} + 
    37018076625)/2$& -60 
    \end{tabularx}
    \caption{The found rational points on $X(\mathrm{b}5,\mathrm{ns}7)/w_5$.}
    \label{tableC}
\end{table}
\begin{remark}\label{rk55}
The curve $C$ has genus 2 and $J(C)(\Q)$ has rank 2, making it infeasible to prove that $C(\Q)=\{Q_1,\ldots,Q_6\}$ using abelian Chabauty. A quadratic Chabauty method, however, might be successful, c.f.\ \cite{quadraticsiksek}, \cite{bbbmtv} and \cite{balakrishnan2021quadratic}. 
\end{remark}
\subsection{Relative symmetric Chabauty}\label{deg4chabsec}
%
In this section, we describe concretely the symmetric Chabauty method introduced in Section \ref{infmwgpsec}, in the case of our interest, where $m=4$, $d=2$ and $K=\Q$. We will use this to determine the subsets $\mathcal{M}_p$ (defined in Section \ref{mwsieve}) in the Mordell--Weil sieve. Consider the notation introduced in Section \ref{infmwgpsec}. In particular, we are given a map $\rho\colon X\to C$ of curves over $\Q$, a prime $p$ of good reduction for both and minimal proper regular models $\mathcal{X}/\Z_p$ and $\mathcal{C}/\Z_p$ for $X_{\Q_p}$ and $C_{\Q_p}$ respectively. We distinguish three kinds of known points $\mathcal{Q}\in X^{(4)}(\Q)$:
\begin{itemize}
    \item[(1)] \emph{pullbacks}: points in  $\rho^*C^{(2)}(\Q)$,
    \item[(2)] \emph{partial pullbacks}: points of the form $\mathcal{P}+\rho^*R$ for $R\in C(\Q)$ and $\mathcal{P}\in X^{(2)}(\Q)$ that are not a pullback, and
    \item[(3)] \emph{isolated points}: points that are neither a pullback nor a partial pullback.
\end{itemize}
Define $\widetilde{X}\colonequals \mathcal{X}_{\F_p}$, $\widetilde{\Omega}\colonequals \Omega_{\widetilde{X}/\F_p}$, and let $V \subset H^0(X_{\Q_p},\Omega)$ be the space of vanishing differentials and $V_C\subset V$ its subspace with trace zero to $C$ (defined in Definition \ref{vandiffdef}). Let $\widetilde{V}$ and $\widetilde{V}_C$ be the images of $V\cap H^0(\mathcal{X},\Omega_{\mathcal{X}/\Z_p})$, resp. $V_C\cap H^0(\mathcal{X},\Omega_{\mathcal{X}/\Z_p})$, under the (surjective) reduction map $H^0(\mathcal{X},\Omega_{\mathcal{X}/\Z_p})\to H^0(\widetilde{X},\widetilde{\Omega})$ (c.f.\ \cite[Lemma 3.6]{box}). Let $\omega_1,\ldots,\omega_{n_C}$ be a basis for $\widetilde{V}_C$, and extend this to a basis $\omega_1,\ldots,\omega_n$ of $\widetilde{V}$, where $n\geq n_C$. Consider a point $\mathcal{Q}\in X^{(4)}(\Q)$, and write $\mathcal{Q}\colonequals \sum_{i=1}^k m_iQ_i$, where $Q_1,\ldots,Q_k\in X$ are distinct points, each $m_i\geq 1$ and $k\leq 4$. For each $i$, let $t_i$ be a uniformiser at $\widetilde{Q}_i$. Denote by $\om_j=\sum_{\ell\geq 0}a_{\ell}(\om_j,t_i)t_i^{\ell}$ the expansion of $\om_j$ at $\widetilde{Q}_i$, and define \[
v_{ij}\colonequals (a_0(\om_j,t_i),\ldots,a_{m_i-1}(\om_j,t_i)).
\]
Now consider the matrices
\[
\mathcal{A}\colonequals (v_{ij})_{\substack{ 1 \leq i \leq k \\ 1\leq j \leq n}} \text{ and } \mathcal{A}_C\colonequals (v_{ij})_{\substack{ 1 \leq i \leq k \\ 1\leq j \leq n_C}}.
\]
Note that $\mathcal{A}$ and $\mathcal{A}_C$ are the reductions (modulo a prime above $p$) of matrices defined in terms of a basis for $V\cap \Omega_{\mathcal{X}/\Z_p}$, resp. $V_C\cap \Omega_{\mathcal{X}/\Z_p}$, and the coefficients $a_i$ appear in the tiny integrals. 
\begin{theorem}[$\square$--Gajovi\'c--Goodman]\label{chabthm} Suppose that $p\geq 17$. 
\begin{itemize}
    \item[(1)]If $\mathcal{Q}$ is a pullback and $\mathrm{rk}(\mathcal{A}_C)\geq 2$, then $D(\mathcal{Q})\cap X^{(4)}(\Q)\subset \rho^*C^{(2)}(\Q)$.
    \item[(2)] If $\mathcal{Q}=\mathcal{P}+\rho^*R$ is a partial pullback and $\mathrm{rk}(\mathcal{A}_C)\geq 3$, then  $D(\mathcal{Q})\cap X^{(4)}(\Q)\subset \mathcal{P}+\rho^*C(\Q)$. 
    \item[(3)] If $\mathrm{rk}(\mathcal{A})\geq 4$, then $D(\mathcal{Q})\cap X^{(4)}(\Q)=\{\mathcal{Q}\}$. 
\end{itemize}
\end{theorem}
\begin{remark}
In case (3), $\mathcal{Q}$ is automatically an isolated point. We also note that 2, 3, and 4 are the largest possible ranks of these matrices respectively, and that $\mathrm{dim}(V)\geq 4$ is a necessary condition for part (3). This is satisfied when $g_X-r_X\geq 4$, where $g_X$ is the genus of $X$ and $r_X$ the rank of $J(X)(\Q)$.
\end{remark}

Recall that $D(\mathcal{Q})\subset X^{(4)}(\Q_p)$ denotes the mod $p$ residue disc of $\mathcal{Q}$. Consider now the special case where   $w:X\to X$ is an involution, $C=X/\langle w\rangle$ and $\rho$ is the quotient map $X\to C$. Suppose moreover that $\mathrm{rk}(J(X)(\Q))=\mathrm{rk}(J(C)(\Q))$. Then (see e.g.\ \cite[Section 3.4]{box})
\[
\mathrm{Ker}(1+\widetilde{w}^*)\subset \widetilde{V}_C\subset\widetilde{V},
\]
where $\widetilde{w}^*: \widetilde{\Omega}\to \widetilde{\Omega}$ is the pullback under $w$. We can thus compute (a submatrix of) $\mathcal{A}$ and $\mathcal{A}_C$ in terms of $\widetilde{X}/\F_p$ directly, without the need to compute $V\cap H^0(\mathcal{X},\Omega)$.  

\subsection{The Mordell--Weil group}\label{Xb5ns7mwgp}
As $\mathrm{rk}(J(\mathrm{b}5,\mathrm{ns}7)(\Q))=\mathrm{rk}(J(C)(\Q))=2$ (see \cite{derickx}), we are indeed in the situation described below Theorem \ref{chabthm}, and we can verify the rank condition of Theorem \ref{chabthm} explicitly for $X=X(\mathrm{b}5,\mathrm{ns}7)$ and $C=X/w_5$. Note that indeed $g_X-r_X\geq 4$.

Before applying the Mordell--Weil sieve, however, we need knowledge of the Mordell--Weil group and a list of known points. In \cite{derickx}, it was shown that $J(\mathrm{b}5,\mathrm{ns}7)(\Q)\simeq \Z/7\Z\times \Z \times \Z$, where the torsion subgroup $\Z/7\Z$ is generated by the difference $D_{\mathrm{tor}}\colonequals c_0-c_{\infty}$ of the two degree 3 Galois orbits of cusps.  Since $J(C)(\Q)$ is hyperelliptic of genus 2, we can use an algorithm of Stoll \cite{stoll} to compute  generators
\[
d_1\colonequals [(-1 , -2 )-\infty^-] \text{ and } d_2\colonequals [(-1/2,-7/2)-\infty^-]
\]
for $J(C)(\Q)\simeq \Z^2$. Pulling these back under $\rho$, we obtain $D_1=\rho^*d_1$ and $D_2=\rho^*d_2$. Finally, by Proposition \ref{boxprop}, the group 
\[
G=\Z/7\Z\cdot  D_{\mathrm{tor}}\oplus \Z \cdot D_1\oplus \Z \cdot D_2
\]
satisfies $2\cdot J(\mathrm{b}5,\mathrm{ns}7)(\Q)\subset G$. We find that the reductions of $D_2$ and $D_1+D_2$ mod 19 are not doubles in $J(\mathrm{b}5,\mathrm{ns}7)(\F_{19})$, so $D_2$ and $D_1+D_2$ also cannot be doubles. 

 We define the maps $\iota: X(\mathrm{b}5,\mathrm{ns}7)^{(4)}(\Q)\to J(\mathrm{b}5,\mathrm{ns}7)(\Q)$ given by $\iota(D)=2\cdot [D-D_0]$, where $D_0=\rho^*d_0$ and $d_0=\infty^++\infty^-$, and
\[
\phi: \Z/7\Z\times \Z^2\longrightarrow J(\mathrm{b}5,\mathrm{ns}7)(\Q),\;\; (a,b,c)\mapsto aD_{\mathrm{tor}}+bD_1+2cD_2. 
\]
Then by the previous discussion, $\mathrm{Im}(\iota)\subset \mathrm{Im}(\phi)$, c.f.\ (iii) in Section \ref{mwsieve}.

\subsection{Finding the points} \label{ptssec}
We first consider the pullbacks and partial pullbacks, and then describe how we found $P_0,\ldots,P_4\in X(\mathrm{b}5,\mathrm{ns}7)$. The curve $X(\mathrm{b}5,\mathrm{ns}7)$ has finitely many quadratic points, so there are finitely many partial pullbacks, whereas the pullbacks $\rho^*C^{(2)}(\Q)$ form an infinite set. For the Mordell--Weil sieve to work, we need to find all isolated points and partial pullbacks, and use a sufficiently large finite subset of $\rho^*C^{(2)}(\Q)$. In this case, however, we can speed up the sieve by discarding all pullbacks thanks to the work of Derickx, Najman and Siksek. 

%
\begin{proposition}
Suppose that $D\in X(\mathrm{b}5,\mathrm{ns}7)^{(4)}(\Q)$ satisfies $w^*(2[D-D_0])=2[D-D_0]$. Then $D\in \rho^*C^{(2)}(\Q)$. 
\end{proposition}
\begin{proof}
Derickx, Najman and Siksek \cite{derickx} showed that $X(\mathrm{b}5,\mathrm{ns}7)$ has no degree 4 map to $\P^1$, other than $\rho$ followed by the hyperelliptic covering map $x$ (and the composition of $x\circ \rho$ with automorphisms of $\P^1$). In particular, each degree 4 map $f:X(\mathrm{b}5,\mathrm{ns}7)\to \P^1$ satisfies $f\circ w_5=f$. 

Since the torsion subgroup has odd size, we find $w_5^*[D-D_0]=[D-D_0]$. As $w_5^*D_0=D_0$ by construction, we conclude that $w_5^*D\sim D$. If $D\neq w_5^*D$, there exists a non-constant $f\in L(D)$ satisfying $\mathrm{div}_0(f)=w_5^*D$ and $\mathrm{div}_{\infty}(f)=D$. Then $f$ defines a degree 4 map $\overline{f}:\;X(\mathrm{b}5,\mathrm{ns}7)\to \P^1$ such that $\overline{f}\circ w_5\neq \overline{f}$. We conclude that $D=w_5^*D$. Now $D\in \rho^*C^{(2)}(\Q)$ unless one of the points in the support of $D$ is fixed by $w_5$. We compute the fixed points of $w_5$ explicitly in \texttt{Magma}, and find that each of them is defined over a field of degree 6. 
\end{proof}
%
Conversely, note that all pullbacks $D=\rho^*d$ also satisfy $w^*(D-D_0)=(D-D_0)$. Define $W=\Z/7\Z\times \Z^2\setminus \{0\}\times \Z^2$ and $S=X(\mathrm{b}5,\mathrm{ns}7)^{(4)}(\Q)\setminus \rho^*C^{(2)}(\Q)$. Since $w_5^*D_{\mathrm{tor}}=-D_{\mathrm{tor}}$, $w_5^*D_1=D_1$ and $w_5^*D_2=D_2$, this proposition means that $\iota(S)\subset \phi(W)$. In practice, this means that we will be able to disregard pullbacks altogether and sieve in $W$, c.f.\ Section \ref{mwsieve}.

Next, we recall that Derickx, Najman and Siksek already found that $X(\mathrm{b}5,\mathrm{ns}7)^{(3)}(\Q)=\{c_0,c_{\infty}\}$; in particular, the curve has no rational points. The quotient $C$ has (at least) 6 rational points $Q_1,\ldots,Q_6$, pulling back to 6 quadratic points. It remains to find the isolated quadratic and quartic points on $X(\mathrm{b}5,\mathrm{ns}7)$.  To this end, we use \texttt{Magma} code written by \"Ozman and Siksek \cite{ozman} to find the irreducible components of the intersections of $X(\mathrm{b}5,\mathrm{ns}7)$ with hyperplanes in $\P^2$, ranging over all hyperplanes $au+bv+cw=0$, where $a,b,c\in \{-5,\ldots,5\}$. As it turns out, we only find the isolated quadratic points $P_0$ and $w_5(P_0)$ this way, not the quartic points. Increasing the search parameters far beyond 5 is too time-consuming, so we need a better search method.

Instead, write $T=\{Q_1,\ldots,Q_6\}\subset C(\Q)$, and define \[
\widetilde{\mathcal{L}}'=(\mathcal{P}_0+\rho^*T)\cup (w_5^*\mathcal{P}_0+\rho^*T) \text{ and }
\widetilde{\mathcal{L}}=(\mathcal{P}_0+\rho^*C(\Q))\cup (w_5^*\mathcal{P}_0+\rho^*C(\Q)),
\]
where $\mathcal{P}_0$ was defined in Theorem \ref{Xb5ns7thm}.
In the next section, we describe the input of the Mordell--Weil sieve for $X(\mathrm{b}5,\mathrm{ns}7)^{(4)}(\Q)$. We can run this sieve with the input as given there, replacing only  $\mathcal{L}$ by $\mathcal{\widetilde{L}}$ and $\mathcal{L}'$ by $\mathcal{\widetilde{L}}'$. The input set of known points being too small, the sieve naturally does not terminate: a number of cosets of $\Z/7\Z\times \Z^2$ remain as possible images of unknown effective degree 4 divisors. These are cosets with respect to a subgroup of huge index (after considering $n$ primes $p_1,\ldots p_n$, the subgroup is contained in the kernel of $\phi_{p_1},\ldots,\phi_{p_n}$, in the notation of Section \ref{mwsieve}), making it statistically unlikely that these cosets contain any elements $(a,b,c)\in \Z/7\Z\times \Z^2$ of small size $\mathrm{max}\{|b|,|c|\}\leq 100$, unless they are there for a reason. Using the LLL-algorithm, we do find elements of small size in some of these cosets. By computing Riemann--Roch spaces (c.f.\ Lemma \ref{lem2.3}), we determine the effective degree 4 divisors mapping to such elements under $\iota$, and voil\`a, we find  $\mathcal{P}_1,\ldots,\mathcal{P}_4$ (as defined in Theorem \ref{Xb5ns7thm}). The sets we thus obtain are
\[
\mathcal{L}'\colonequals \{\mathcal{P}_1,w_5^*\mathcal{P}_1,\ldots,\mathcal{P}_4,w_5^*\mathcal{P}_4\}\cup (\mathcal{P}_0+\rho^*T)\cup (w_5^*\mathcal{P}_0+\rho^*T),
\]
and 
\[
\mathcal{L}\colonequals \{\mathcal{P}_1,w_5^*\mathcal{P}_1,\ldots,\mathcal{P}_4,w_5^*\mathcal{P}_4\}\cup (\mathcal{P}_0+\rho^*C(\Q))\cup (w_5^*\mathcal{P}_0+\rho^*C(\Q))\supset \mathcal{L}'
\]
\subsection{The proof of Theorem \ref{Xb5ns7thm}}\label{Xb5ns7thmproofsec}
We apply the Mordell--Weil sieve defined in Section \ref{mwsieve} with $V=X(\mathrm{b}5,\mathrm{ns}7)^{(4)}$, $\mathcal{A}=J(\mathrm{b}5,\mathrm{ns}7)$, $B=\Z/7\Z\times \Z^2$, $\phi$ and $\iota$ as defined at the end of Section \ref{Xb5ns7mwgp} and $S$, $W$ and $\mathcal{L}$ as defined in Section \ref{ptssec}. For each prime $p\geq 17$, we define $\mathcal{N}_p\subset \widetilde{X}(\mathrm{b}5,\mathrm{ns}7)(\F_p)$ to be the reductions of the points in $\mathcal{L}'$ satisfying part (2) or part (3) of Theorem \ref{chabthm}. When $p<17$, let $\mathcal{N}_p=\emptyset$. Then we set
\[
\mathcal{M}_p=\iota_p^{-1}(\phi_p(W))\setminus \mathcal{N}_p.
\]
All assumptions for the sieve are then satisfied, and we run it with the primes $p_1=11$, $p_2=13$, $p_3=17$, $p_4=23$, $p_5=53$, $p_6=29$, $p_7=71$, $p_8=43$, $p_9=37$ and $p_{10}=31$. This terminates, and by Proposition \ref{mwsieveprop} we have thus proved Theorem \ref{Xb5ns7thm}. 
\subsection{The $j$-invariants}
On the planar Derickx--Najman--Siksek model for $X(\mathrm{b}5,\mathrm{ns}7)$ as described, the $j$-invariant map is not known. Instead, to compute the $j$-invariant of the points in Table \ref{tableb5ns7}, we use \cite{boxmodels}. Here a different (canonical) model for $X(\mathrm{b}5,\mathrm{ns}7)$ is computed, together with the morphism from this model to our planar model, and with the $j$-invariant. We thus pull back the effective degree 4 divisors of Table \ref{tableb5ns7} to this canonical model, and evaluate the $j$-invariant there. 

\section{$X(\mathrm{b}3,\mathrm{b}5,\mathrm{e}7)$ and $X(\mathrm{s}3,\mathrm{b}5,\mathrm{e}7)$}\label{X1andX2sec}
By Corollary \ref{X1X2cor} and subsequent remarks, it suffices now to determine the quadratic points on 
\[
X_1\colonequals X(\mathrm{b}5,\mathrm{e}7)/w_5 \text{ and } X_2\colonequals X(\mathrm{b}5,\mathrm{e}7)/\phi_7w_5.
\]
Recall from Section \ref{sec51} that both curves admit a degree 2 map $\rho_i:X_i\to C=X(\mathrm{b}5,\mathrm{ns}7)/w_5$. We also recall that $C$ is a hyperelliptic curve of genus 2 whose Mordell--Weil group has rank 2, so determining $C(\Q)$ with abelian Chabauty is unlikely to succeed, c.f.\ Remark \ref{rk55}.  Consequently, we cannot expect to be able to determine $\rho_i^*C(\Q)\subset X_i^{(2)}(\Q)$ using abelian Chabauty. Instead, therefore, we shall attempt to describe $X_i^{(2)}(\Q)\setminus \rho_i^*C(\Q)$ using, again, relative symmetric Chabauty and the Mordell--Weil sieve. This suffices, as rational points on $C$ come from rational or quadratic points on $X(\mathrm{b}5,\mathrm{ns}7)$ and therefore have rational or quadratic $j$-invariants. 
\subsection{Overview} The curves $X_1$ and $X_2$ have genera 8 and 5 respectively. 

 Unlike for $X(\mathrm{b}5,\mathrm{ns}7)$, each $J(X_i)(\Q)$ appears to have rank strictly greater than $2=\mathrm{rk}(J(C)(\Q))$. Compared to $X(\mathrm{b}5,\mathrm{ns}7)$, this causes two extra difficulties when attempting to use Chabauty:
\begin{itemize}
    \item We do not obtain a finite index subgroup of $J(X_i)(\Q)$ by pulling back generators of $J(C)(\Q)$, making it hard to sieve effectively.
    \item We do not obtain vanishing differentials with trace zero from the kernel of $1+w_i^*$, where $w_i$ is the involution on $X_i$ such that $C=X_i/w_i$. 
\end{itemize}
We instead use cusp forms to find vanishing differentials in Section \ref{vandiffsec}, while the Mordell--Weil group issue is addressed in the remainder of this section. \\

Recall from  Section \ref{mwsieve} that rather than the Jacobian, all we really need for sieving is a morphism $X_i^{(2)}\to A$ for some abelian variety $A$. As a start, we can try to take $A=J(C)$. Then $X_i^{(2)}\to J(C)$ is the Abel--Jacobi map composed with the push-forward $(\rho_i)_*$. Sieving in $J(C)(\Q)$ is unlikely to work, however, because both $\#X_i^{(2)}(\F_p)$ and $\#J(C)(\F_p)$ have size close to $p^2$, c.f.\ Remark \ref{sieveproportion}. Which other Jacobians, of which we know the Mordell--Weil group, does $X_i^{(2)}$ map to?

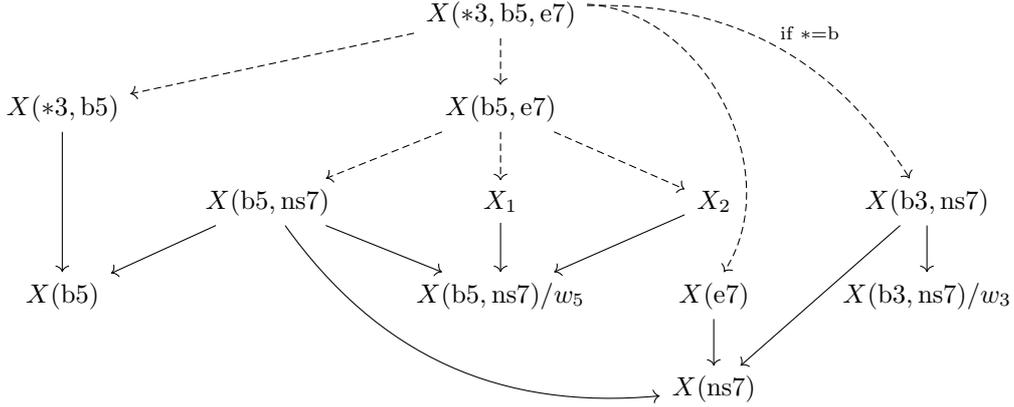
\begin{figure}[h!]
\begin{tikzcd}
                               &                                                                              & {X(*3,\mathrm{b}5,\mathrm{e}7)} \arrow[lld, dashed] \arrow[rrdd, "\text{if }*=\mathrm{b}", dashed, bend left] \arrow[d, dashed] \arrow[rddd, dashed, bend left=60] &                          &                                                     \\
{X(*3,\mathrm{b}5)} \arrow[dd] &                                                                              & {X(\mathrm{b}5,\mathrm{e}7)} \arrow[ld, dashed] \arrow[d, dashed] \arrow[rd, dashed]                                                                               &                          &                                                     \\
                               & {X(\mathrm{b}5,\mathrm{ns}7)} \arrow[ld] \arrow[rd] \arrow[rrdd, bend right] & X_1 \arrow[d]                                                                                                                                                      & X_2 \arrow[ld]           & {X(\mathrm{b}3,\mathrm{ns}7)} \arrow[ldd] \arrow[d] \\
X(\mathrm{b}5)                 &                                                                              & {X(\mathrm{b}5,\mathrm{ns}7)/w_5}                                                                                                                                  & X(\mathrm{e}7) \arrow[d] & {X(\mathrm{b}3,\mathrm{ns}7)/w_3}                   \\
                               &                                                                              &                                                                                                                                                                    & X(\mathrm{ns}7)          &                                                    
\end{tikzcd}
\label{fig1}
\caption{Maps between modular curves, where $*\in\{\mathrm{b},\mathrm{s}\}$. A full line means we have computed this map explicitly, whereas a dotted line has not been computed. }

\end{figure}

One answer is $J(\mathrm{e}7)$. The curve $X(\mathrm{e}7)$ is an elliptic curve of rank 0, and $X_i^{(2)}$ has a map to its Jacobian $J(\mathrm{e}7)$. Indeed, each degree 2 effective divisor $D\in X_i^{(2)}$ pulls back to a degree 4 effective divisor on $X(\mathrm{b}5,\mathrm{e}7)$, then pushes forward to a degree 4 effective divisor on $X(\mathrm{e}7)$. This then has an Abel--Jacobi map into $J(\mathrm{e}7)$. In particular, demanding that each $D\in X_i^{(2)}(\F_p)$ maps into the image of $J(\mathrm{e}7)(\Q)$ in $J(\mathrm{e}7)(\F_p)$, we obtain a stronger sieve (with $A=J(C)\times J(\mathrm{e}7)$). Now $\mathrm{dim}(A)=3>2=\mathrm{dim}(X_i^{(2)})$, making $A$ potentially large enough for sieving. In practice, however, this sieve still does not terminate fast enough. 

We thus use the final trump card we have been dealt but neglected so far: the information at the prime 3. We are only interested in those quadratic points on $X_i$ that are the image of a quartic point on $X(\mathrm{s}3,\mathrm{b}5,\mathrm{e}7)$ or $X(\mathrm{b}3,\mathrm{b}5,\mathrm{e}7)$. The curves $X(\mathrm{s}3,\mathrm{b}5)$ and $X(\mathrm{b}3,\mathrm{b5})=X_0(15)$ are also elliptic curves of rank 0, and we can make use of their Mordell--Weil groups. More precisely, we apply the sieve from Section \ref{mwsieve} to the four varieties
\[
Y_{*,i}\colonequals X_i^{(2)}\times_{X(\mathrm{b}5)^{(4)}}X(*3,\mathrm{b}5)^{(4)} \quad\text{ for } i\in \{1,2\},\; *\in \{\mathrm{b},\mathrm{s}\}
\]
using the abelian variety $A=J(C)\times J(\mathrm{e}7)\times J(*3,\mathrm{b}5)$ and a relative symmetric Chabauty method. Note that the map $X_i^{(2)}\to X(\mathrm{b}5)^{(4)}$ is defined by first pushing forward to $C$, then pulling back to $X(\mathrm{b}5,\mathrm{ns}7)$ and then pushing forward to $X(\mathrm{b}5)$, see Figure \ref{fig1}. This sieve is successful for three of the four cases: for $i=1$ and $*=\mathrm{b}$, the sieve does not terminate. 

The reason is that $C$ is hyperelliptic. Denote by $x:C\to \P^1$ the map to $\P^1$. All $D\in x^*\P^1(\Q)\subset C^{(2)}(\Q)$  have the same image $O\in J(C)(\Q)$. So in order to ``remove'' in the sieve a class $(O,a,b)\in J(C)\times J(\mathrm{e}7)\times J(\mathrm{b}3,\mathrm{b}5)$ as a possible image of an unknown element of $X_1^{(2)}(\Q)$, we need, modulo a prime $p$, all $p+1$ such pullbacks $x^*\P^1(\F_p)$ to be discounted simultaneously by the sieve; for one such class $(O,a,b)$ this does not happen when $i=1$ and $*=\mathrm{b}$. 

We therefore add yet another curve and yet another Mordell--Weil group. The curve $X(\mathrm{b}3,\mathrm{ns}7)$ has genus 5, and its quotient $C_2\colonequals X(\mathrm{b}3,\mathrm{ns}7)/w_3$ is a genus 2 hyperelliptic curve whose Mordell--Weil group is isomorphic to $\Z/2\Z\times \Z^2$. We thus redefine
\[
Y_{\mathrm{b},1}\colonequals X(\mathrm{b}3,\mathrm{ns}7)^{(4)}\times_{X(\mathrm{ns}7)^{(4)}}X_1^{(2)}\times_{X(\mathrm{b}5)^{(4)}}X(\mathrm{b}3,\mathrm{b}5)^{(4)}
\]
and sieve in $A=J(C_2)\times J(C)\times J(\mathrm{e}7)\times J(\mathrm{b}3,\mathrm{b}5)$. This works. 

\begin{theorem}\label{fibredproductthm}
The sets $Y_{*,i}(\Q)$, for $*\in \{\mathrm{b},\mathrm{s}\}$ and $i\in\{1,2\}$, consist entirely of points of which the image in $X_i^{(2)}(\Q)$ is in $\rho_i^*C(\Q)$. 
\end{theorem}
 Together with Corollary \ref{X1X2cor} and subsequent remarks, Theorem \ref{fibredproductthm} shows that all  quartic points on $X(\mathrm{b}3,\mathrm{b}5,\mathrm{e}7)$ and $X(\mathrm{s}3,\mathrm{b}5,\mathrm{e}7)$ either have quadratic or rational $j$-invariant, correspond to a $\Q$-curve, or have non-totally real $j$-invariant displayed in Table \ref{tableb5ns7}. Together with Theorems \ref{X0105thm} and \ref{Xs3b5b7thm}, this finally proves Theorem \ref{thm1.5}.

To prove Theorem \ref{fibredproductthm}, we note that it remains necessary to use a relative Chabauty method on top of the sieve, because each $Q\in C(\Q)$ determines for each $i\in \{1,2\}$ a point on $Y_{*,i}(\Q)$ for at least one $*\in \{\mathrm{b},\mathrm{s}\}$. In the next section, we describe the symmetric Chabauty method used to prove Theorem \ref{fibredproductthm} (in conjunction with the sieve).

\subsection{Relative symmetric Chabauty for quadratic points}\label{deg2chabsec}
We consider again the situation as in Section \ref{infmwgpsec}, with $K=\Q$: consider a prime $p$, curves $X/\Q$ and $C/\Q$ with minimal proper regular models $\mathcal{X}/\Z_p$ and $\mathcal{C}/\Z_p$ respectively, and a $\Z_p$-morphism $\rho: \mathcal{X}\to \mathcal{C}$ of degree 2 on generic fibres. The difference is that we are now interested in $X^{(2)}(\Q)$ instead of $X^{(4)}(\Q)$.

We consider again the space $V_C$ of vanishing differentials with trace zero on $C$. Let $\om_1,\ldots,\om_n\in V_C\cap H^0(\mathcal{X},\Omega^1_{\mathcal{X},\Z_p})$ be a basis for $V_C$, and consider $\mathcal{Q}\in X^{(2)}(\Q)$ and a uniformiser $t_Q$ at a point $Q\in \mathcal{Q}$, such that $t_Q$ reduces to a uniformiser at the reduction $\widetilde{Q}$ of $Q$ at a prime $\pp$ of $\Q(Q)$ above $p$.  In the ring $\widehat{\Omega}_{X,Q}$, we then have
\[
\om_i=(a_0(i,Q)+a_1(i,Q)t_Q+a_2(i,Q)t_Q^2+\ldots)\dd t_Q,\;\; a_j(i,Q)\in \overline{\Z}_p \text{ for each }i,j,Q,
\]
and $\widetilde{\om}_i=(\sum_{j=0}^{\infty} \widetilde{a}_j(i,Q)\widetilde{t}_Q)\dd \widetilde{t}_Q$, where a tilde denotes reduction modulo $\pp$. 

The following theorem is due to Siksek \cite{siksek}.
\begin{theorem}[Siksek]\label{siksekthm} Assume that $p\geq 3$ is a prime of good reduction for $X$ and $C$.
If $\mathcal{Q}=\rho^*R$ for some $R\in C(\Q)$ and there exist $i\in \{1,\ldots,n\}$ and $Q\in \mathcal{Q}$ such that
\[
\widetilde{a}_0(i,Q)\neq 0,
\]
then each $\mathcal{P}\in X^{(2)}(\Q)$ in the residue disc of $\mathcal{Q}$ is also in $\rho^*C(\Q)$. 
\end{theorem}
This theorem allows us to define the sets $\mathcal{M}_p$ required for the Mordell--Weil sieve, c.f.\ Section \ref{mwsieve}. In the next sections, we describe how to obtain the remaining sieve input, such as  the models of these curves, the maps between them, the Mordell--Weil groups and the vanishing differentials. 

\subsection{Models and maps}\label{modmapsec}
In this section, we give explicit descriptions for the models and maps in Figure \ref{fig1}, as well as the known Mordell--Weil groups.

\subsubsection{The map $X(\mathrm{b}3,\mathrm{b}5)\to X(\mathrm{b}5)$}
We begin on the left-hand side. A model for the elliptic curve $X(\mathrm{b}3,\mathrm{b}5)=X_0(15)$, together with the map to $X(\mathrm{b}5)=X_0(5)=\P^1$, can be computed by built-in functions of the Small Modular Curves package in \texttt{Magma}: 
\[
X(\mathrm{b}3,\mathrm{b}5):\;y^2 + xy + y = x^3 + x^2 - 10x - 10,
\]
and the degree 4 map to $X(\mathrm{b}5)=\P^1$ is given by
\[
(x:y:z)\mapsto (-6x^2 + xy + 6xz + 10yz + 12z^2 :
x^2 + 2xz + z^2),
\]
where $(x:y:z)$ are the coordinates in $\P^2$. The Mordell--Weil group is
\[
J(\mathrm{b}3,\mathrm{b}5)(\Q)=\Z/2\Z \cdot (-1, 0)\oplus \Z/4\Z \cdot (-2,-2).
\]

\subsubsection{The map $X(\mathrm{s}3,\mathrm{b}5)\to X(\mathrm{b}5)$}
The curve $X(\mathrm{s}3,\mathrm{b}5)$ is isomorphic to $X_0(45)/w_9$ (c.f.\ \cite[Example 2.10]{boxmodels}), for which the \texttt{ModularCurveQuotient} function in \texttt{Magma} can compute a model using modular forms. However, for such quotients, the map to $X_0(5)$ is not built in. 

Instead, we compute a canonical model for the non-hyperelliptic genus 3 curve $X_0(45)$ in $\P^2$, by finding equations between three linearly independent weight 2 cusp forms $f_1,f_2,f_3$. On $X_0(45)$, we compute a matrix for the action of $w_9$ on these cusp forms (and hence on this model), and by taking the quotient, we obtain explicitly the map $X_0(45)\to X_0(45)/w_9$, and find
\[
X(\mathrm{s}3,\mathrm{b}5):\; y^2 + xy + y = x^3 + x^2 - 5x + 2.
\]
This map describes $x$ and $y$ as rational functions in terms of $f_1,f_2$ and $f_3$, so we obtain $q$-expansions $x(q)$, $y(q)$ for the functions $x$ and $y$ on $X(\mathrm{s}3,\mathrm{b}5)$. Now using the \texttt{qExpansionsOfGenerators} function in \texttt{Magma}, we find the $q$-expansion $h(q)$ of a Hauptmodul on $X(\mathrm{b}5)$. 

Finally, we compute $X(\mathrm{s}3,\mathrm{b}5)\to X(\mathrm{b}5)$ using the method of Section \ref{modelssec}. We do not display this degree 6 map here for brevity reasons, but invite the curious reader to use the \texttt{Magma} code. 

The Mordell--Weil group is
\[
J(\mathrm{s}3,\mathrm{b}5)(\Q)=\Z/2\Z \cdot (3/4, -7/8)\oplus \Z/4\Z \cdot (0,-2).
\]
\subsubsection{The map $X(\mathrm{e}7)\to X(\mathrm{ns}7)$} This map was determined by Freitas, Le Hung and Siksek \cite{freitas} using the $j$-invariant map on $X(\mathrm{ns}7)=\P^1$ computed by Chen \cite{chen}. They found $X(\mathrm{e}7)$ to be an elliptic curve given by
\[
X(\mathrm{e}7):\; y^2=7(16x^4 + 68x^3 + 111x^2 + 62x + 11)
\]
with the map $X(\mathrm{e}7)\to X(\mathrm{ns}7)$ being simply a projection $(x:y:z)\mapsto (x:z)$. 

The Mordell--Weil group is
\[
J(\mathrm{e}7)(\Q)=\Z/2\Z\cdot [(-1/3 , -14/3 )-(-1/3 , 14/3 )]
\]

\subsubsection{$X_1$, $X_2$ and $X(\mathrm{b}5,\mathrm{ns}7)$}
Canonical models for these curves, as well as the map $X(\mathrm{b}5,\mathrm{ns}7)\to X(\mathrm{ns}7)$, were computed by the author in \cite[Section 5]{boxmodels}. We display them here for $X(\mathrm{b}5,\mathrm{ns}7)$ and $X_2$. The model for $X_1$ in $\P^7$ is an intersection of 15 quadrics, which were displayed \cite{boxmodels}. 

A canonical model for $X(\mathrm{b}5,\mathrm{ns}7)$ in $\P_{X_0,\ldots,X_5}^5$ is given by:
\begin{align*}X(\mathrm{b}5,\mathrm{ns}7):\; & 14X_0^2 + 12X_2X_3 - 16X_3^2 - 14X_2X_4 + 30X_3X_4 - 11X_4^2\\&\;\;\;\;\; + 28X_2X_5 - 58X_3X_5 + 40X_4X_5 - 28X_5^2=0,\\& 7X_0X_1 - 2X_2X_4 - 4X_3X_4 + 2X_4^2 + 12X_3X_5 - 7X_4X_5 + 10X_5^2=0,\\& 14X_1^2 - 4X_2X_3 + 16X_3^2 + 10X_2X_4 + 14X_3X_4 - 21X_4^2\\&\;\;\;\;\; + 4X_2X_5 - 58X_3X_5 + 64X_4X_5 - 66X_5^2=0,\\& 2X_0X_2 - 2X_0X_3 + 2X_1X_3 - 5X_0X_4 - 6X_1X_4 + 8X_0X_5 + 4X_1X_5=0,\\& 4X_1X_2 - 2X_0X_3 - 6X_1X_3 - X_0X_4 + 3X_1X_4 + 3X_0X_5 - 2X_1X_5=0,\\& 8X_2^2 - 20X_2X_3 + 16X_3^2 - 14X_2X_4 + 14X_3X_4 - 21X_4^2\\&\;\;\;\;\; + 28X_2X_5 - 42X_3X_5 + 56X_4X_5 - 28X_5^2=0.
\end{align*}  
and the Atkin--Lehner involution acts as
\[
w_5:\; (X_0:X_1:X_2:X_3:X_4:X_5)\mapsto (-X_0:-X_1:X_2:X_3:X_4:X_5).
\]
The map $X(\mathrm{b}5,\mathrm{ns}7)\to X(\mathrm{ns}7)=\P^1$ is given by
\[
(X_0:\ldots:X_5)\mapsto (7X_0 - 2X_2 + 4X_3 - X_4 - 4X_5:
    -14X_0 - 7X_1 + 6X_2 - 12X_4 + 10X_4 - 9X_5).
\]
A canonical model for $X(\mathrm{b}5,\mathrm{e}7)/\phi_7w_5$ in $\P^4_{X_0,\ldots,X_4}$ is given by

\begin{align*} X(\mathrm{b}5,\mathrm{e}7)/\phi_7w_5:\;& 448X_0^2 - 9X_1^2 + 9X_2^2 + 54X_2X_3 + 9X_3^2 + 112X_0X_4 + 126X_1X_4 + 7X_4^2=0,\\& 16X_0X_1 - 3X_1^2 + 3X_2^2 + 6X_2X_3 + 3X_3^2 + 2X_1X_4 + 21X_4^2=0,\\& 3X_1X_2 + 28X_0X_3 + 12X_1X_3 + 21X_2X_4 + 14X_3X_4=0
\end{align*}
and the remaining involution acts by
\[
w_5=\phi_7:\; (X_0:X_1:X_2:X_3:X_4)\mapsto (X_0:X_1:-X_2:-X_3:X_4).
\]
The quotients of each of these three curves by their respective involutions yield the map to $C$. The map to $X(\mathrm{ns}7)$ was also computed using Chen's $j$-map and is therefore compatible with $X(\mathrm{e}7)\to X(\mathrm{ns}7)$. This $j$-map was used to find the $q$-expansion for a Hauptmodul on $X(\mathrm{ns}7)$, after which the map $X(\mathrm{b}5,\mathrm{ns}7)\to X(\mathrm{ns7})$ was determined using the method described in Section \ref{modelssec}. This strategy did not succeed, however, in determining the map $X(\mathrm{b}5,\mathrm{ns}7)\to X(\mathrm{b}5)$, since the degree of this map (which is 21) is too large for such a computation. Instead, note that the map to $X(\mathrm{b}5)=\P^1$ corresponds to an element $g\in \Q(X(\mathrm{b}5,\mathrm{ns}7))$. Denote by $j\in \Q(X(\mathrm{b}5,\mathrm{ns}7))$ the function corresponding to the $j$-invariant, computed via the map to $X(\mathrm{ns}7)$ and Chen's $j$-map on $X(\mathrm{ns}7)$. The $j$-invariant $j_5$ on $X(\mathrm{b}5)$ is a rational function of degree 6 such that
\[
j_5(g)-j=0 \in \Q(X(\mathrm{b}5,\mathrm{ns}7))(g). 
\]
Multiplying this equation by the denominator of $j_5(g)$, we obtain a degree 6 polynomial equation in $g$, of which $\texttt{Magma}$ can find the single root in $\Q(X(\mathrm{b}5,\mathrm{ns}7))$. This root yields the map $X(\mathrm{b}5,\mathrm{ns}7)\to X(\mathrm{b}5)$. We do not display the map here for brevity reasons. 

\subsubsection{The curve $X(\mathrm{b}3,\mathrm{ns}7)$} This curve has genus 5. We computed a model, as well as the action of $w_3$, via cusp forms using the algorithm described in \cite{boxmodels}. This yields the following model in $\P^4$:
\begin{align*}
10528X_0^2 - 21X_1X_2 - 112X_2^2 - 1136X_0X_3 - 17X_3^2 - 1016X_0X_4 + 273X_3X_4 - 176X_4^2&=0,\\ 47X_1^2 + 58X_1X_2 - 4X_2^2 - 336X_0X_3 - 9X_3^2 - 560X_0X_4 - 2X_3X_4 - 60X_4^2&=0,\\ 40X_0X_1 + 8X_0X_2 + X_1X_3 + 3X_2X_3 - 7X_1X_4&=0.
\end{align*}
We find that the Atkin--Lehner involution $w_3$ acts by
\[
w_3:\; (X_0:\ldots:X_4)\mapsto (X_0:-X_1:-X_2:X_3:X_4).
\]
The quotient by this automorphism yields the genus 2 hyperelliptic curve
\[
C_2=X(\mathrm{b}3,\mathrm{ns}7)/w_3:\;y^2 = x^6 + 6x^5 - x^4 - 46x^3 - 43x^2 - 12x.
\]
Using an algorithm of Stoll \cite{stoll}, we compute its Mordell--Weil group $J(X(\mathrm{b}3,\mathrm{ns}7)/w_3)(\Q)$: 
\[
\Z\cdot [(-1 , 3 )-\infty^-]\oplus\Z\cdot [(0,0)-\infty^-]\oplus \Z/2\Z\cdot [(0,0)-(-4 , 0 )].
\]
Finally, we use the $q$-expansions of the cusp forms used to compute our model for $X(\mathrm{b}3,\mathrm{ns}7)$ and of the Hauptmodul on  $X(\mathrm{ns}7)$ to compute $X(\mathrm{b}3,\mathrm{ns}7)\to X(\mathrm{ns}7)$ via the method of Section \ref{modelssec}. This yields
\[
(X_0:\ldots:X_4)\mapsto (3X_1 + 6X_2 + 5X_3 + 6X_4:
56X_0 - 6X_1 - 4X_2 - 6X_3 + 4X_4).
\]
\begin{remark}\label{Xb3ns7remark}
While we will use the Mordell--Weil group $J(C_2)(\Q)$ for sieving, we could also instead use a subgroup $G\subset J(\mathrm{b}3,\mathrm{ns}7)(\Q)$ satisfying $2\cdot J(\mathrm{b}3,\mathrm{ns}7)(\Q) \subset G$, c.f.\ Remark \ref{mwrk}. This would result in a stronger sieve. We can construct such $G$ as follows. By \cite[Theorem 1]{chen}, $J(\mathrm{b}3,\mathrm{ns}7)$ is isogenous to the new part of $J(X(\mathrm{b}3,\mathrm{b}49)/w_{49})$. In \texttt{Magma}, we can compute the newforms corresponding to this new part of the Jacobian, and find that it is isogenous to $A_f\times A_g\times A_h$, where $f,g,h$ are newforms, $A_f=J(C_2)$, and $A_g$ and $A_h$ have dimensions 1 and 2 respectively. Here $A_k$ is the $\Q$-simple abelian variety attached to the newform $k$. We compute that $L(g,1)\neq 0$ and $L(h,1)\neq 0$, which by Kolyvagin--Logach\"ev \cite{kolyvagin} implies that $A_g(\Q)$ and $A_h(\Q)$ are torsion, so that $J(\mathrm{b}3,\mathrm{ns}7)(\Q)$ and $J(C_2)(\Q)$ have equal rank. Pulling back $J(C_2)(\Q)$ to $J(\mathrm{b}3,\mathrm{ns}7)$ under the quotient map thus yields a subgroup $H\subset J(\mathrm{b}3,\mathrm{ns}7)(\Q)$ satisfyng $2\cdot J(\mathrm{b}3,\mathrm{ns}7)(\Q)\subset H$ modulo torsion by Proposition \ref{boxprop}. Finally, composing the map $X(\mathrm{b}3,\mathrm{ns}7)\to X(\mathrm{ns}7)$ with the $j$-map on the latter, we find the $j$-map on $X(\mathrm{b}3,\mathrm{ns}7)$. Its poles consist of two irreducible degree 3 divisors $c_0$ and $c_{\infty}$: the cuspidal divisors. The difference $[c_0-c_{\infty}]$ has order 7 in the Mordell--Weil group, and reduction modulo primes shows that $2\cdot J(\mathrm{b}3,\mathrm{ns}7)(\Q)_{\mathrm{tors}}\subset \langle [c_0-c_{\infty}]\rangle$. The group $G$ generated by $[c_0-c_{\infty}]$ and the pullbacks of the generators of $J(C_2)(\Q)$ displayed above thus satisfies $2\cdot J(\mathrm{b}3,\mathrm{ns}7)\subset G$.
\end{remark}
\subsection{The vanishing differentials on $X_1$ and $X_2$}\label{vandiffsec} In \cite[Section 5]{boxmodels}, we defined the following newforms of trivial Nebentypus character:
\begin{align*}
   f_{49}&\colonequals  q+q^2-q^4-3q^8-3q^9+O(q^{11}) \in S_2(\Gamma_0(49)\cap \Gamma_1(7),\overline{\Q}), \\
   f_{35}&\colonequals q+q^3-2q^4-q^5+q^6-2q^8-3q^{10}+O(q^{11})\in S_2(\Gamma_0(5)\cap \Gamma_1(7),\overline{\Q}),\\
   g_{35}&\colonequals q+ \al q^2 -(\al+1)q^3 +(2-\al)q^4+ q^5 -4q^6, -q^7+ (\al - 4)q^8+(\al + 2)q^9\\&\;\;\;\;\;\;\;\;\;\;+ \al q^{10}+O(q^{11}) \in S_2(\Gamma_0(5)\cap \Gamma_1(7),\overline{\Q}), \text{ where }\al=(-1+\sqrt{17})/2,\\
   f_0&\colonequals q -2q^2 -3q^3+ 2q^4 +q^5+ 6q^6 +6q^9 -2q^{10}+O(q^{11})\in S_2(\Gamma_0(5\cdot 7^2)\cap \Gamma_1(7),\overline{\Q}), \\
   f_1&\colonequals q+\sqrt{2}q^2-(\sqrt{2}+1)q^3-q^5-(\sqrt{2}+2)q^6-2\sqrt{2}q^8+2\sqrt{2}q^9\\&\;\;\;\;\;\;\;\;\;\;-\sqrt{2}q^{10}+O(q^{11}) \in S_2(\Gamma_0(5\cdot 7^2)\cap \Gamma_1(7),\overline{\Q}), \text{ and }\\
   f_2&\colonequals q+ (1+\sqrt{2})q^2+ (1-\sqrt{2})q^3+(2\sqrt{2}+1)q^4 +q^5-q^6+ (\sqrt{2}+3)q^8-2\sqrt{2}q^9\\&\;\;\;\;\;\;\;\;\;\;+(1+\sqrt{2})q^{10}+O(q^{11})\in S_2(\Gamma_0(5\cdot 7^2)\cap \Gamma_1(7),\overline{\Q}).
\end{align*}
Let $G(\mathrm{b}5,\mathrm{e}7)\subset \mathrm{GL}_2(\Z/35\Z)$ be the intersection of the inverse images of $G(\mathrm{e}7)\subset \mathrm{GL}_2(\F_7)$ and $B_0(5)\subset \mathrm{GL}_2(\F_5)$ under the reduction maps. Define the fixed spaces \[
S_1=S_2(\Gamma_0(5\cdot 7^2)\cap \Gamma_1(7),\Q(\zeta_7)^+)^{\langle G(\mathrm{b}5,\mathrm{e}7),w_5\rangle}\text{ and }S_2=S_2(\Gamma_0(5\cdot 7^2)\cap \Gamma_1(7),\Q(\zeta_7)^+)^{\langle G(\mathrm{b}5,\mathrm{e}7),\phi_7w_5\rangle},
\]
where the action of $G(\mathrm{b}5,\mathrm{e}7)$ was defined in \cite[Section 3.1]{boxmodels}. Despite containing cusp forms with Fourier coefficients in $\Q(\zeta_7)^+$, both $S_1$ and $S_2$ are $\Q$-vector spaces. Then the map $f(q)\mapsto f(q)(\dd q)/q$ defines isomorphisms $S_i\simeq H^0(X_i,\Omega)$ for $i\in \{1,2\}$. We computed in \cite{boxmodels} canonical models for $X_1$ and $X_2$ by finding equations satisfied by the $q$-expansions of bases for $S_1$ and $S_2$ respectively.  We describe the isomorphism $S_i\simeq H^0(X_i,\Omega^1)$ explicitly. 
\begin{lemma}\label{canonicalisolemma}
Suppose that $X$ is a non-hyperelliptic curve of genus at least 2, and $f_0,\ldots,f_n$ is a basis for $H^0(X,\Omega^1)$. We obtain a canonical map $\phi:X\to \P^n$, $x\mapsto (f_0(x):\ldots:f_n(x))$ which is an isomorphism onto its image $Z$. Consider $Q\in X$ with local coordinate $q$ at $Q$. Then for each $i$ we have a power series expansion $f_i(q)=(\sum_{j\geq 1}a_{ij}q^j)(\dd q)/q$ at $Q$. Let $x_0,\ldots,x_n$ be the variables of $\P^n$, and consider $\om\in H^0(Z,\Omega^1)$. Then  $\omega=g(x_0,\ldots,x_n)\dd (x_0/x_n)$, where $g$ is a quotient of homogeneous polynomials of equal degree. The expansion of $\phi^*\om$ at $Q$ is
\[
\phi^*\om(q) = g(f_0(q),\ldots,f_n(q))\frac{\dd (f_0/f_n)}{\dd q}\dd q.
\]
\end{lemma}
\begin{proof}
This follows directly from the definitions.
\end{proof}
Given our model for $X_i$, we compute a basis for $H^0(X_i,\Omega)$, leading to an explicit isomorphism $H^0(X_i,\Omega)\simeq S_i$.  We use this to determine the 1-form on $X_i$ corresponding to a given $q$-expansion of a cusp form in $S_i$.

Denote by $B_d$ the operator on modular forms mapping $q\mapsto q^d$ and by $\chi$ the Dirichlet character $\chi:(\Z/7\Z)^{\times}\to \CC^{\times}$ given by $\chi(3)=-e^{2\pi i/3}$. For $X_2$, the chosen basis of $S_2$ was $h_0,h_1,h_2,h_3,h_4$, where
\begin{align*}
h_0&\in V_{49}\colonequals \mathrm{Span}_{\overline{\Q}}\{f_{49},f_{49}\otimes \chi^2,f_{49}\otimes \chi^4,f_{49}|B_5,(f_{49}\otimes \chi^2)|B_5,(f_{49}\otimes \chi^4)|B_5\}\\
h_1&\in \mathrm{Span}_{\overline{\Q}}\{f_{35},f_{35}\otimes \chi^2,f_{35}\otimes \chi^4,f_{35}|B_7\}, \quad h_4\in \mathrm{Span}_{\overline{\Q}}\{f_0\otimes \chi,f_0\otimes \chi^3,f_0\otimes \chi^5\}.
\end{align*}
 Here, when $f$ is a newform and $\ep$ a Dirichlet character, we denote by $f\otimes \ep$ the newform (at some level) satisfying $a_n(f\otimes\ep)=a_n(f)\ep(n)$ for all $n$ coprime to the level of $f$ and the conductor of $\ep$. 
 The forms $h_2$ and $h_3$ are fixed by the remaining involution on $X_2$ and thus correspond to the 1-forms on the quotient $C$. They are linear combinations of $f_1,f_1\otimes \chi^2, f_1\otimes \chi^4$ and their Galois conjugates. 

For $X_1$, the chosen basis of $S_1$ was $g_0,g_1,\ldots,g_7$, where $g_4=h_2$, $g_5=h_3$, and
\begin{align*}
g_0&\in V_{49},\; g_1,g_2\in \mathrm{cSpan}_{\overline{\Q}}\{g_{35},g_{35}\otimes \chi^2,g_{35}\otimes \chi^4,g_{35}|B_7\},\;g_3\in \mathrm{Span}_{\overline{\Q}}\{f_0,f_0\otimes \chi^2,f_0\otimes\chi^4\},\\
 g_6&,g_7\in \mathrm{cSpan}_{\overline{\Q}}\{f_2,f_2\otimes \chi^2,f_2\otimes\chi^4,(f_2\otimes \chi^2)|B_7,(f_2\otimes \chi^4)|B_7\},
\end{align*}
where $\mathrm{cSpan}(A)$ denotes the span of $A$ and the Galois conjugates of elements in $A$.

 By studying these twists of eigenforms, we will show that $h_1,h_4$ and $g_1,g_2,g_6$ and $g_7$ correspond to vanishing differentials on $X_2$ and $X_1$ respectively. These differentials can then  be determined explicitly in terms of the models for $X_1$ and $X_2$ using Lemma \ref{canonicalisolemma}. 
 
In \cite[Section 5]{boxmodels}, it was shown that the modular curve
\[
Y\colonequals X(\Gamma_0(5\cdot 7^2)\cap \Gamma_1(7))
\]
admits a morphism $\pi_i:Y\to X_i$ for each $i\in \{1,2\}$. While both $Y$ and $X_i$ are defined over $\Q$, this morphism is defined over $K=\Q(\zeta_7)^+$, the real subfield of $\Q(\zeta_7)$. Moreover, $H^0(Y_K,\Omega^1)\simeq S_2(\Gamma_0(5\cdot 7^2)\cap \Gamma_1(7),K)$, and the embedding $\pi_i^*: H^0((X_i)_K,\Omega^1)\to H^0(Y_K,\Omega^1)$ corresponds to the inclusion map on modular forms
\[
S_i\otimes K\subset S_2(\Gamma_0(5\cdot 7^2)\cap \Gamma_1(7),K). 
\]
\begin{lemma}\label{map1formlemma}
Let $K$ be a number field, and choose a prime $\pp$ of $\O_K$. Suppose that $\pi: X\to Z$ is a non-constant morphism of curves over $K$, and $\om\in H^0(Z,\Omega^1)$ is such that $\pi^*\om$ is a vanishing differential. Then $\om$ is a vanishing differential.
\end{lemma}
\begin{proof}
Consider $D\in J(Z)(K)$. We first note that $\pi_*(J(X)(K))\subset J(Z)(K)$ has finite index (for $D\in J(Z)(K)$, we have $\mathrm{deg}(\pi)D=\pi_*\pi^*D\in \pi_*J(X)(K)$). So after multiplying by an integer, we may suppose that $D=\pi_*E$ for some $E\in J(X)(K)$. Then 
\[
\int_{\pi_*E}\om =\int_E\pi^*\om =0
\]
by the chain rule (\ref{chainrule}). 
\end{proof}

If we can find $f\in S_i$ corresponding to a vanishing differential on $Y_K$, then $f$ corresponds to a vanishing differential on $(X_i)_K$ (hence on $X_i/\Q$) by the previous lemma. To this end, we make use of rank 0 quotients. 
Now let $g$ be an eigenform in $S_2(\Gamma_0(5\cdot 7^2)\cap \Gamma_1(7),\overline{\Q})$. By Eichler--Shimura theory, we obtain a morphism $\pi_g\colon J(Y)\to A_g$, where $A_g$ is the Abelian variety associated to $g$ by Eichler and Shimura. Let $\iota\colon Y\to J(Y)$ be the Abel--Jacobi map. Then $\iota^*\pi_g^*H^0((A_g)_{\overline{\Q}},\Omega)$ is the space of 1-forms generated by the forms corresponding to $g$ and its $\mathrm{Gal}(\overline{\Q}/\Q)$-conjugates.  Let us assume that $A_g(K)$ has rank zero. Then each element in $A_g(K)$ is torsion, and for each $\om \in H^0(A_g,\Omega^1)$ and each $D\in J(Y)(K)$, we thus have an equality of Coleman integrals
\[
\int_D\iota^*\pi_g^*\om = \int_{(\pi_g)_*D}\om =0
\]
by the chain rule (\ref{chainrule}) and because torsion elements annihilate $H^0(A_g,\Omega)$. To prove that $g$ corresponds to a vanishing differential on $Y_K$, it thus suffices to show that $A_g(K)$ has rank 0, for which we use Proposition \ref{kolyvaginprop}. Note that the group of characters on $\mathrm{Gal}(K/\Q)$ is generated by $\chi^2$. 

\begin{corollary}\label{vandiffcor}
The cusp forms $h_1$ and $h_4$ correspond to vanishing differentials on $X_2$ with trace zero to $C$. The forms $g_1,g_2,g_6$ and $g_7$ correspond to vanishing differentials on $X_1$ with trace zero to $C$.
\end{corollary}
\begin{proof} Let $\Psi\colonequals \{h_1,h_4,g_1,g_2,g_6,g_7\}$. Each element of $\Psi$ is a linear combination of the newforms in $\Phi\colonequals  \{f_{35},f_{35}\otimes \chi^2,f_{35}\otimes \chi^4,g_{35},g_{35}\otimes \chi^2,g_{35}\otimes \chi^4,f_0\otimes\chi^3,f_0\otimes\chi,f_0\otimes\chi^5,f_2,f_2\otimes\chi^2,f_2\otimes\chi^4\}$ and their Galois conjugates. By Lemmas \ref{canonicalisolemma} and  \ref{map1formlemma}, the discussion below the latter and Proposition \ref{kolyvaginprop}, to prove that the cusp forms in $\Psi$ are vanishing differentials, it suffices to show that $f_{35}$, $f_0\otimes \chi^3$, $g_{35}$ and $f_2$ have no CM and no inner twists, and that each $f\in \Phi$ satisfies $L(f,1)\neq 0$. We verified this in \texttt{Sage}, which uses an algorithm of Tim Dokchitser \cite{dokchitser} to evaluate $L(f,1)$. 

It remains to show that each element of $\Psi$ corresponds to a differential on $X_i$ with trace zero to $C$. Consider the map $\pi_C\colon Y\to C$ (which factors via $X_1$ and $X_2$). It suffices to show that each element of $\Phi$ corresponds to a differential $\om$ on $Y$ with $\pi_{C,*}(\om)=0$.  Now $\pi_C^*H^0(C_K,\Omega)$ is the space generated by the 1-forms corresponding to $h_2=g_4$ and $h_3=g_5$, each of which is a linear combination of $f_1, f_1\otimes \chi^2,f_1\otimes \chi^4$ and their conjugates. The pushforward map $\pi_{C,*}\colon J(Y)\to J(C)$ thus factors via 
\[
\pi_{f_1}\colon J(Y)\to A_{f_1}\times A_{f_1\otimes \chi^2}\times A_{f_1\otimes \chi^4}.
\]
As $\Phi$ consists of eigenforms unequal to (conjugates of) $f_1$,  $f_1\otimes\chi^2$ and $f_1\otimes \chi^4$, indeed $\pi_{f_1}(\om)=0$ for each $\om\in H^0(Y,\Omega)$ corresponding to an eigenform in $\Phi$, from which it follows that also $\pi_{C,*}(\om)=0$, as desired. 
\end{proof}
\subsection{Proof of Theorem \ref{fibredproductthm}}
When $(*,i)\in \{(\mathrm{b},2),(\mathrm{s},1),(\mathrm{s},2)\}$, we apply the Mordell--Weil sieve described in Section \ref{mwsieve} with $V=X_i^{(2)}\times_{X(\mathrm{b}5)^{(4)}}X(*3,\mathrm{b}5)^{(4)}$ and $A=J(C)\times J(\mathrm{e}7)\times J(*3,\mathrm{b}5)$. When $(*,i)=(\mathrm{b},1)$, we apply the sieve with $V=  X(\mathrm{b}3,\mathrm{ns}7)^{(4)}\times_{X(\mathrm{ns}7)^{(4)}}X_1^{(2)}\times_{X(\mathrm{b}5)^{(4)}}X(\mathrm{b}3,\mathrm{b}5)^{(4)}$ and $A=J(C)\times J(\mathrm{e}7)\times  J(\mathrm{b}3,\mathrm{b}5)\times J(X(\mathrm{b}3,\mathrm{ns}7)/w_3))$. In all cases, the map $\iota:V\to A$ is determined by the Abel--Jacobi maps on $C^{(2)}$, $X(*3,\mathrm{b}5)^{(4)}$, $X(\mathrm{e}7)^{(4)}$ and $(X(\mathrm{b}3,\mathrm{ns}7)/w_3)^{(4)}$, and we let $\mathcal{L}$ be the set of points in $V(\Q)$ mapping into $\{\rho_i^*(Q_j) \mid j\in \{1,\ldots,6\}\}$ on $X_i^{(2)}$.  (In particular, we find no other rational points on $V$.) When $(*,i)\in \{(\mathrm{b},2),(\mathrm{s},1),(\mathrm{s},2)\}$, we let $B=(\Z^2)\times (\Z/2\Z)\times (\Z/2\Z\times \Z/4\Z)$ and define $\phi\colon B\to A(\Q)$ via the generators for $J(C)(\Q)$, $J(\mathrm{e}7)(\Q)$ and $J(*3,\mathrm{b}5)(\Q)$ given in Section \ref{modmapsec}. Similarly, when $(*,i)=(\mathrm{b},1)$, we take $B=(\Z^2)\times (\Z/2\Z)\times (\Z/2\Z\times \Z/4\Z)\times (\Z/2\Z\times \Z^2)$ and $\phi$ is defined via the generators. 

For a prime $p\geq 3$ of good reduction for $V$, we define $\mathcal{N}_p\subset \widetilde{V}(\F_p)$ to be the reductions of the points in $\mathcal{L}$ whose image in $X_i^{(2)}$ satisfies the non-vanishing condition in Theorem \ref{siksekthm}. By Corollary \ref{vandiffcor}, we can verify this condition using the vanishing differentials corresponding to $h_1,h_4$ (when $i=2$) and $g_1,g_2,g_6$ and $g_7$ (when $i=1$) under the isomorphism of Lemma \ref{map1formlemma}.  Then, as before, $\mathcal{M}_p=\iota_p^{-1}(\phi_p(B))\setminus \mathcal{N}_p$.  In each of the four sieves, the primes  13, 23, 43, 53, 67, 83, 71, 89, 97, 79 and 181 suffice to obtain an empty intersection, so we are done by Proposition \ref{mwsieveprop}.

\section{Further study}\label{furtherstudy}
\subsection{Quartic fields containing $\sqrt{5}$}
When considering quartic fields containing $\sqrt{5}$, we can no longer use Thorne's theorem (Theorem \ref{modliftthm} (ii)). Instead, we can use \cite[Proposition 2.1]{freitas}, to obtain the following theorem.
\begin{theorem}
Suppose that $K$ is any totally real quartic field. If an elliptic curve $E/K$ is not modular, then $E$ gives rise to a $K$-point on one of the curves
\[
X(\mathrm{u}3,\mathrm{v}5,\mathrm{w}7), \quad u\in \{\mathrm{b},\mathrm{s}\},\; v\in \{\mathrm{b},\mathrm{s},\mathrm{ns}\},\; w\in \{\mathrm{b},\mathrm{e}\}.
\]
\end{theorem}
It was noted in \cite[Remark (iii) in Section 4.2]{freitas} that we cannot do any better than this at 5 without stronger modularity lifting results for fields containing $\sqrt{5}$. Instead of 4 curves, we thus need to consider 16 curves, a monumental task. An advantage is that instead of quartic points, we now need to study quadratic points over $\Q(\sqrt{5})$. 

On the other hand, we lose an important $\Q$-curve argument used multiple times in Sections \ref{X0105sec} and \ref{Xs3b5b7sec}: on many curves to consider, only one Atkin--Lehner involution exists. For example, the quotient of $X(\mathrm{b}3,\mathrm{ns}5)$ (genus 2) by $w_3$ is $\P^1$ and the quotient of $X(\mathrm{ns}5,\mathrm{b}7)$  (genus 5) by $w_7$ is an elliptic curve with infinitely many $\Q(\sqrt{5})$-points. These $\Q(\sqrt{5})$-points on the quotients pull back to infinitely many quadratic points over $\Q(\sqrt{5})$ on $X(\mathrm{b}3,\mathrm{ns}5)$ and $X(\mathrm{ns}5,\mathrm{b}7)$. Those points correspond to $\Q(\sqrt{5})$-curves, which, unlike $\Q$-curves, are not known to be modular. 

Another complicating factor is that an analysis of the Jacobians of $X(\mathrm{u}3,\mathrm{v}5,\mathrm{b}7)$ for $u\in \{\mathrm{b},\mathrm{s}\}$ and $v\in \{\mathrm{s},\mathrm{ns}\}$  (using the methods outlined in Section \ref{ranksec}) revealed multiple 1-dimensional factors corresponding to elliptic curves with positive rank over $\Q(\sqrt{5})$, particularly in the $\mathrm{ns}5$ case. 

\subsection{Quintic fields}
After quartic fields, one naturally wonders what is possible for quintic fields. An apparent advantage is that quintic points on modular curves, unlike quartic points, do not arise as inverse images of lower degree points under degree 2 maps, which are abundant on modular curves due to the Atkin--Lehner involutions. 

However, other problems do arise due to the increased degree. The Chabauty method for quartic points on $X(\mathrm{b}5,\mathrm{ns}7)$ does not work for quintic points, as the Chabauty condition $r<g-(d-1)$ is not satisfied when $r=2$, $g=6$ and $d=5$. Other quotients of $X(\mathrm{b}3,\mathrm{b}5,\mathrm{e}7)$ and $X(\mathrm{s}3,\mathrm{b}5,\mathrm{e}7)$ that do satisfy the Chabauty condition and have small genus do not appear to exist. For $X(\mathrm{b}3,\mathrm{b}5,\mathrm{b}7)=X_0(105)$, its handy genus 5 quotient $X_0(105)/w_5$ has infinitely many quintic points (we computed multiple Riemann--Roch spaces to be 2-dimensional, c.f.\ Section \ref{finitemwgpsec}). Similarly, the method for computing quartic points on $X(\mathrm{s}3,\mathrm{b}5,\mathrm{b}7)$ used in Section \ref{Xs3b5b7sec} also fails in practice for quintic points.

While not necessarily infeasible, it is clear that new ideas are required for quintic fields.

\bibliographystyle{alpha}
\bibliography{refs}
\end{document}